\documentclass[10pt]{article}
\usepackage{graphicx}
\usepackage{epsfig}
\usepackage{amssymb}
\usepackage{amsmath}
\usepackage{amsthm}
\newtheorem{theorem}{Theorem}[section]

{\end{description}}
{\end{description}}

{\hfill $\bullet$ \\}

\newcommand{\be}{\begin{equation}}
\newcommand{\ee}{\end{equation}}

\begin{document}
    \title{A Fast Second-Order Explicit Predictor-Corrector Numerical Technique To Investigating And Predicting The Dynamic Of Cytokine Levels And
    Human Immune Cell Activation In\\ Response To Gram-Positive Bacteria: Staphylococcus Aureus}
    \author{Eric Ngondiep$^{\text{\,1}}$, Ariane Njomou Ndantouo$^{\text{\,2,3}}$, George Mondinde Ikomey$^{\text{\,2,3}}$}
   \date{$^{\text{\,1\,}}$\small{Department of Mathematics and Statistics, College of Science, Imam Mohammad Ibn Saud\\ Islamic University
        (IMSIU), $90950$ Riyadh $11632,$ Saudi Arabia.}\\
     \text{\,}\\
       $^{\text{\,2\,}}$\small{School of Health Sciences, Catholic University of Central Africa, $1110$ Messa-Yaounde, Cameroon.}\\
     \text{\,}\\
     $^{\text{\,3\,}}$\small{Center for the Study and Control of Communicable Diseases (CSCCD), Faculty of Medicine\\ and Biomedical Sciences, University of Yaounde 1, $8445$ Yaounde, Cameroon.}\\
     \text{\,}\\
        \textbf{Correspondence email:} ericngondiep@gmail.com/engondiep@imamu.edu.sa}

    \maketitle
   \textbf{Abstract.}
   This paper develops a second-order explicit predictor-corrector numerical approach for solving a mathematical model on the dynamic of cytokine expressions and human immune cell activation in response to the bacterium staphylococcus aureus (S. aureus). The proposed algorithm is at least zero-stable and second-order accurate. Mathematical modeling works that analyze the human body in response to some antigens have predicted concentrations of a broad range of cells and cytokines. This study deals with a coupled cellular-cytokine model which predicts cytokine expressions in response to gram-positive bacteria S. aureus. Tumor necrosis factor alpha, interleukin 6, interleukin 8 and interleukin 10 are included to assess the relationship between cytokine release from macrophages and the concentration of the S. aureus antigen. Ordinary differential equations are used to model cytokine levels while the cellular responses are modeled by partial differential equations. Interactions between both components provide a more robust and complete systems of immune activation. In the numerical simulations, a low concentration of S. aureus is used to measure cellular activation and cytokine expressions. Numerical experiments indicate how the human immune system responds to infections from different pathogens. Furthermore, numerical examples suggest that the new technique is faster and more efficient than a large class of statistical and numerical schemes discussed in the literature for systems of nonlinear equations and can serve as a robust tool for the integration of general systems of initial-boundary value problems.
   \text{\,} \\
    \text{\,}\\
   \ \noindent {\bf Keywords: mathematical model on dynamic of mixed cellular-cytokine, cytokine expressions, human immune cell activation, S. aureus, a second-order explicit predictor-corrector numerical technique, numerical simulations.} \\
   \\
   {\bf AMS Subject Classification (MSC). 65Lxx}.

      \section{Introduction and motivation}\label{sec1}
      Mathematical models arising from the complex biological systems are important decision tools that allow to elucidate emergent properties of intricate biological pathways with the human body \cite{1tgq,2tgq,4tgq,3tgq,5tgq}. The mathematical modeling works based on the human immune cells in response to pathogens have predicted the concentrations of a broad range of cells including: B-Cell, dentritic, macrophages and plasma cells. The leukocytes of the immune system are usually the neutrophils, monocytes, basophiles eosinophiles and lymphocytes. They play a very important role in recognizing and fighting infections. Their main function is to eliminate pathogens by phagocytosis and cell debris through engulfment and chemical degradation. The interaction between the bacterium and host during infection leads to immune cell activation and often to a change in the relative ratio of leukocyte sub-populations of the innate immune and adaptive immune responses \cite{6tgq, 1prr}. Raman spectroscopy based on spectroscopic fingerprint in clinical samples are efficient approaches to identify and differentiate these subsets of leukocytes \cite{14prr,16prr,18prr}. Furthermore, this technique also detects immune cells activation and apoptosis \cite{20prr,25prr,23prr,22prr}.\\

      Adaptive immunity deals with a long-term specific response initiated to eliminate a well-known pathogen whereas innate immunity considers many lines of defense starting with: saliva, skin, various secretions and ending with non-specific leukocytes \cite{7tgq,8tgq}. Additionally, the innate immune system is capable to recognize and tag a wide set of antigens. The inability of innate immune response to kill the pathogen leads to the activation of the adaptive immune system, which is firstly composed with T cells and B cells. The functions of B cells is to produce  antibodies to neutralize the pathogen and  subsequently destroy them. Furthermore, T and B cells contribute in the production of small signaling proteins released by leukocytes, so called cytokines, which facilitate the communication between immune cells. These cytokines include: tumor necrosis factor alpha (TNF$\alpha$), interleukin 6 (IL6), interleukin 8 (IL8) and interleukin 10 (IL10) \cite{9tgq,10tgq}. TNF$\alpha$ is fundamental to the acute phase reaction during inflammatory response. IL6 is a pro-inflammatory cytokine that is involved in inflammation and homeostasis processes and can also act as an anti-inflammatory cytokine through its inhibitory effects on TNF$\alpha$. In addition, IL6 plays a crucial role in the recruitment of T cells and in the production of T and B cells during inflammation and the delay in apoptosis of T cells \cite{11tgq}. In the literature \cite{12tgq,13tgq,14tgq,15tgq}, the authors have shown that IL8 is induced by TNF$\alpha$, inhibited by IL10, that is an anti-inflammatory cytokine critical in the regulation of immune response and IL8 is also involved in the recruitment of basophils, neutrophils and T cells. Moreover, IL10 maintains homeostasis and prevents host damage during infection while acting as an immuno-regulator. Furthermore, this anti-inflammatory cytokine limits the production of pro-inflammatory cytokine including IL6, while down-regulating the expression of TNF$\alpha$, T helper type 1 cytokines, and major histocompatibility complex class 2 molecules \cite{11tgq,17tgq}. IL6, IL8 and IL10 exert many effects on the immune system, hematopoiesis and acute-phase response. In \cite{15kztw,19kztw}, the authors showed that monocytes, lymphocytes, cancer cells as well as macrophages have been documented for the production and secretion of IL6, IL8 and IL10. These cytokines in a para-crime manner induce in vitro growth of melanoma cells, prostate, lung, ovarian, kidney and cervical cancers. Specifically, the disease level and the activity of immune system mechanisms modulated by some interleukins such as: IL6, IL8 and IL10, allow to evaluate the efficiency of the treatment and prognosis in course of malignancy. For more details, we refer the readers to \cite{9kztw,16kztw,6kztw,4kztw,10kztw}.\\

       Although present in all bacterial cell walls, peptidoglycan (PePG) helps to stabilize cell structure and shape. PePG protects cells from bursting in response to environmental stressors. Similarly, lipoteichoic acid (LTA), known as a gram-positive bacterium, is composed with a glycolipid covalently bound and a hydrophilic glycerophosphate polymer. Innate immune system deals with peptidoglycan and lipoteichoic acid which are able to trigger the systemic release of cytokines \cite{18tgq,19tgq}. Lipopolysaccharide (LPS), a hallmark of gram-negative bacteria, commonly called endotoxin, consists of a membrane-anchor lipid, glycan polymer and oligosaccharide core. It lies in the class of the most potent immuno-stimulants. It has been proven \cite{20tgq,21tgq} that pro-inflammatory and anti-inflammatory cytokine activation in gram-negative bacterial infections are mainly driven by endotoxin. Analogously, one cause of the pro-inflammatory cytokine activity follows from the physiological recognition of the lipid component by the immune system \cite{22tgq,23tgq,24tgq,21tgq}. PePG and LTA are the main sources of activation of cytokines in response to gram-positive bacterial infections whose levels can be predicted by the use of mathematical models \cite{19tgq,25tgq,4tgq,26tgq,27tgq}.

      Complex systems of differential equations arising in mathematical biology and physics play a major role in various branches of science and engineering by describing their interaction among the diffusion transport and reaction mechanism. A large class of unsteady nonlinear differential equations such as: the  mathematical models of covid-19, Navier-Stokes problems, Stokes-Darcy models, convection-diffusion-reaction equations, time-fractional equations, advection-diffusion problem, mathematical models of dynamic of poverty and corruption, etc.., are introduced to capture some features of several real-world biological and physical phenomena \cite{1en,30tgq,2en,31tgq,3en,4en,32tgq,5en}. For example: traffic flow, plasma physics, fluid mechanics, optical fibers, chemical kinetics, mathematical immunology, population dynamics, neutron nuclear reaction, financial derivatives pricing \cite{3xws,2xws}. A wide range of applications requires the analytical solutions of such equations, which allow to well understand the interesting features, novel phenomena and intrinsic mechanism hidden in the dynamic systems \cite{4xws}. Since the exact solutions of these equations only exist under some restrict conditions, develop efficient and reliable numerical techniques for such systems of unsteady nonlinear equations is of great interest. In the literature, researchers have described and analyzed abundant statistical and numerical approaches in an approximate solutions \cite{9krbf,13krbf,9en,10en,17krbf,18krbf,11en,13en,20krbf,14en,21krbf,28tgq,33tgq,34tgq}. Most of the existing mathematical models of complex biological systems deal with the cytokine response to lipopolysaccharide and the basic pathway of the standard immune response to S. aureus. Although these models represent responses to several pathogens such as: gram-positive and gram-negative bacteria, the cytokine expressions are computed as functions of activated macrophages. This suggests similar interactions between cytokines and macrophages \cite{35tgq,36tgq,28tgq}.\\

       In this paper, the mathematical model considers a system of ordinary differential equations (ODEs) coupled with partial differential ones (PDEs) to emulating the relationship between the pathogen S. aureus, cytokines and cells with the human immune system. Specifically, a second-order explicit predictor-corrector scheme is developed to investigating and predicting the cells-based activation and cytokine expressions induced by the bacterium S. aureus. Both theoretical and numerical results indicate that the proposed algorithm is less time consuming, faster and more efficient than a large class of statistical and numerical approaches applied to such complex systems of differential equations \cite{18krbf,29tgq,7en,17krbf,8en,20krbf}. In addition, our results show that the developed numerical method can be observed as a robust tool to predict ex vivo and in vivo experimental data induced by a given antigen. The highlights of this work are the following:
      \begin{description}
        \item[i.] mathematical formulation of the dynamic of cytokine concentrations and activation of human immune cells in response to the bacterium S. aureus,
        \item[ii.] construction of the explicit predictor-corrector scheme and analysis of stability and convergence rate of the proposed algorithm,
        \item[iii.] numerical examples to confirm the theory.
      \end{description}

      The paper is organized as follows: we present the mathematical model on the dynamic of cytokine levels and human immune cell activation in response to the propagation of S. aureus in Section $\ref{sec2}$. In Section $\ref{sec3}$, we develop the two-step explicit numerical technique for solving the considered problem and we provide both stability analysis and convergence order of the new algorithm. Section $\ref{sec4}$ considers some numerical simulations whereas the general conclusions and future works are described in Section $\ref{sec5}$.

      \section{Mathematical formulation of the model}\label{sec2}
      This section deals with a coupled inflammatory cytokine model and cellular one as proposed in \cite{28tgq,26tgq} (see also Figure $\ref{fig c1}$). This combination is so called mixed cellular-cytokine model and it represents an extension of some previous works discussed in the literature \cite{4tgq,3tgq,1tgq,30tgq,31tgq} by considering the inter-connectivity between the cytokine and cellular responses and captures ex vivo and potential in vivo. The cellular model given in \cite{36tgq} predicts the relationships between activated and resting macrophages, antibodies and S. aureus whereas the cytokine model analyzes how changes in activated macrophages and resting ones impact cytokine evolution for TNF$\alpha$, IL6, IL8 and IL10 \cite{28tgq,36tgq}. Both models operate in the same time frame whenever the concentrations of the initial bacteria and the activated macrophages are scaled to match those reported in the ex vivo studies \cite{28tgq,12tgq,35tgq}. The coupled cellular-cytokine model is described by a system of PDEs and ODEs and it will be efficiently solved using a new second-order explicit predictor-corrector method. The cellular model is represented by the PDEs while the cytokine one considers the ODEs. It is worth noticing to mention that the gram-positive bacteria (for instance: S. aureus) causes the wide-spread inflammation and septic shock that is primarily due to the function LPA and PePG during an inflammatory response \cite{37tgq,20tgq,11tgq}. The work carried out by LTA and PePG induces the cytokine concentration in the host's innate and adaptive immune response to S. aureus pathogen. In \cite{20tgq,11tgq,37tgq}, researchers have established that the action mechanisms in the host which include: neutrophil flux, phagocytosis and Sbi protein activation differ from the wall components and the given pathogen, but indicate the same inflammatory cytokine responses.

       \begin{figure}
         \begin{center}
          \begin{tabular}{c c}
         \psfig{file=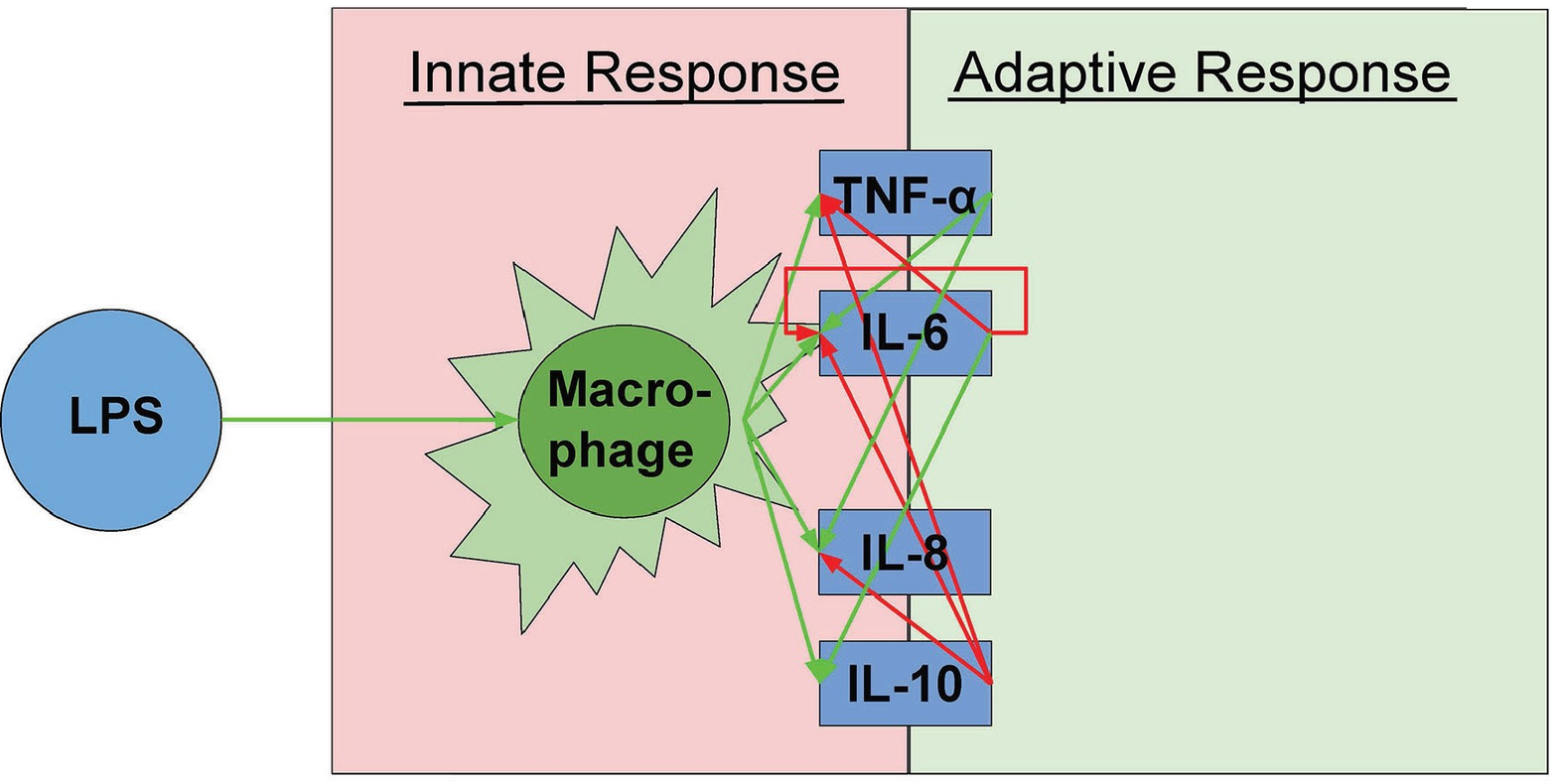,width=5cm} & \psfig{file=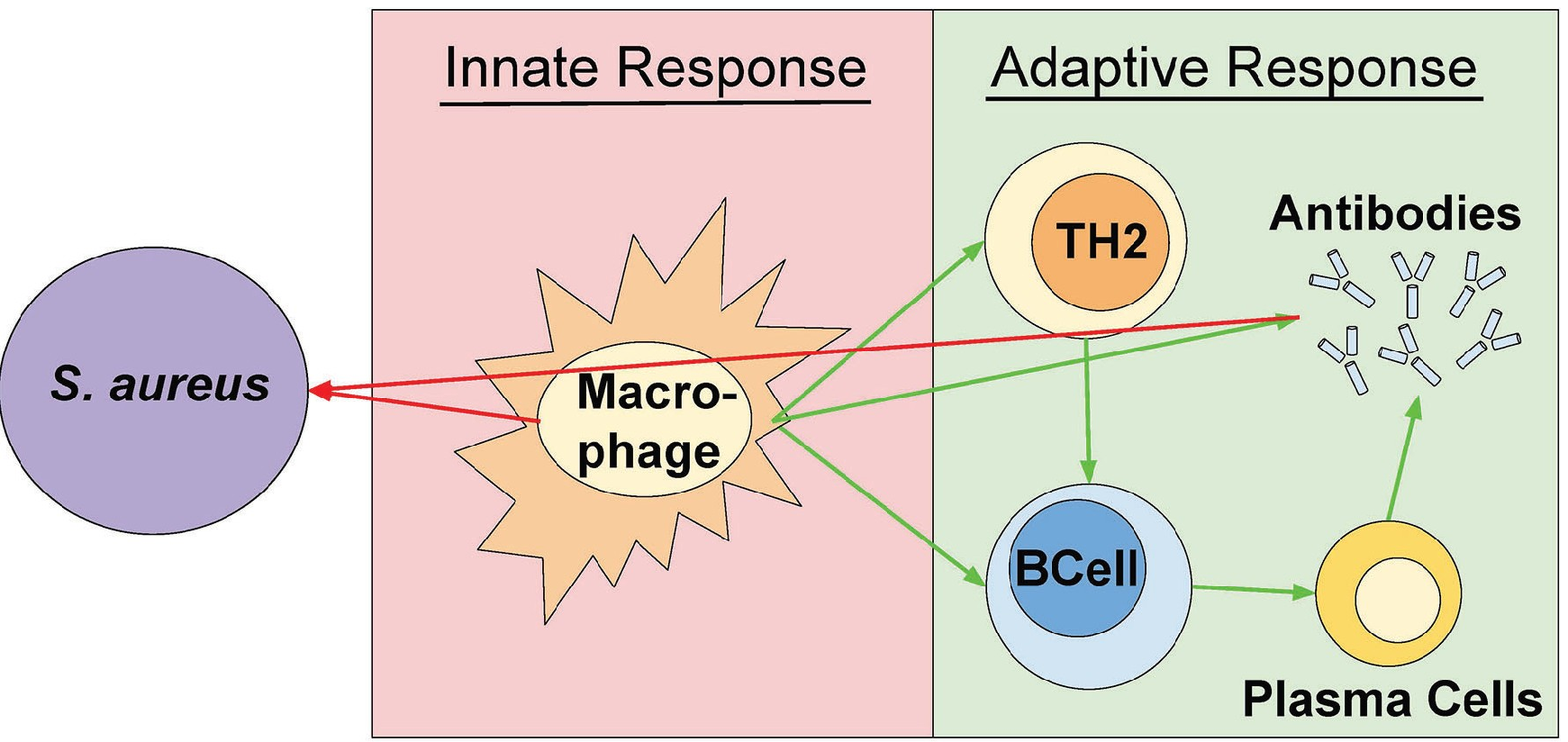,width=5cm} \\
         \psfig{file=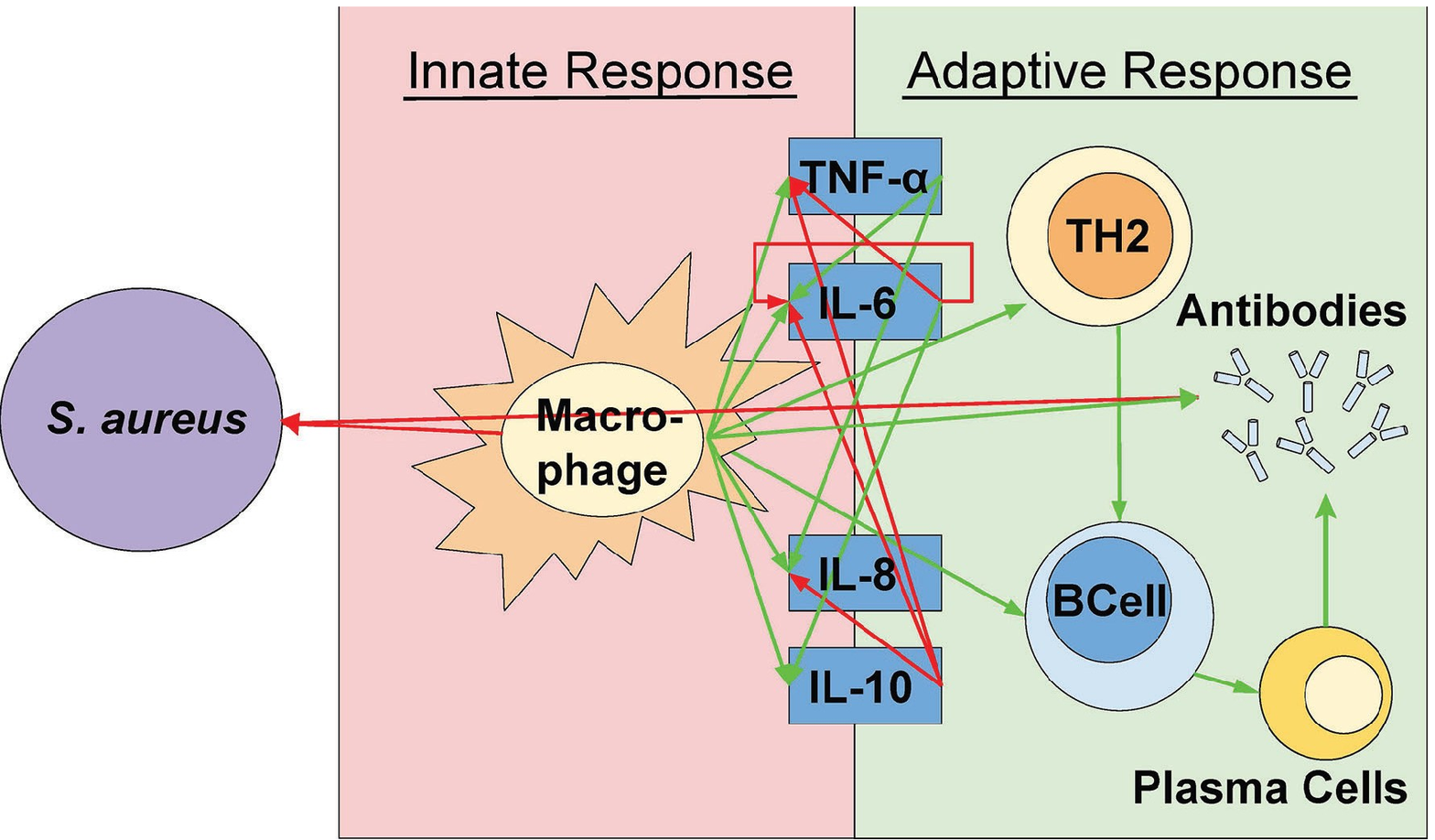,width=5cm} & \psfig{file=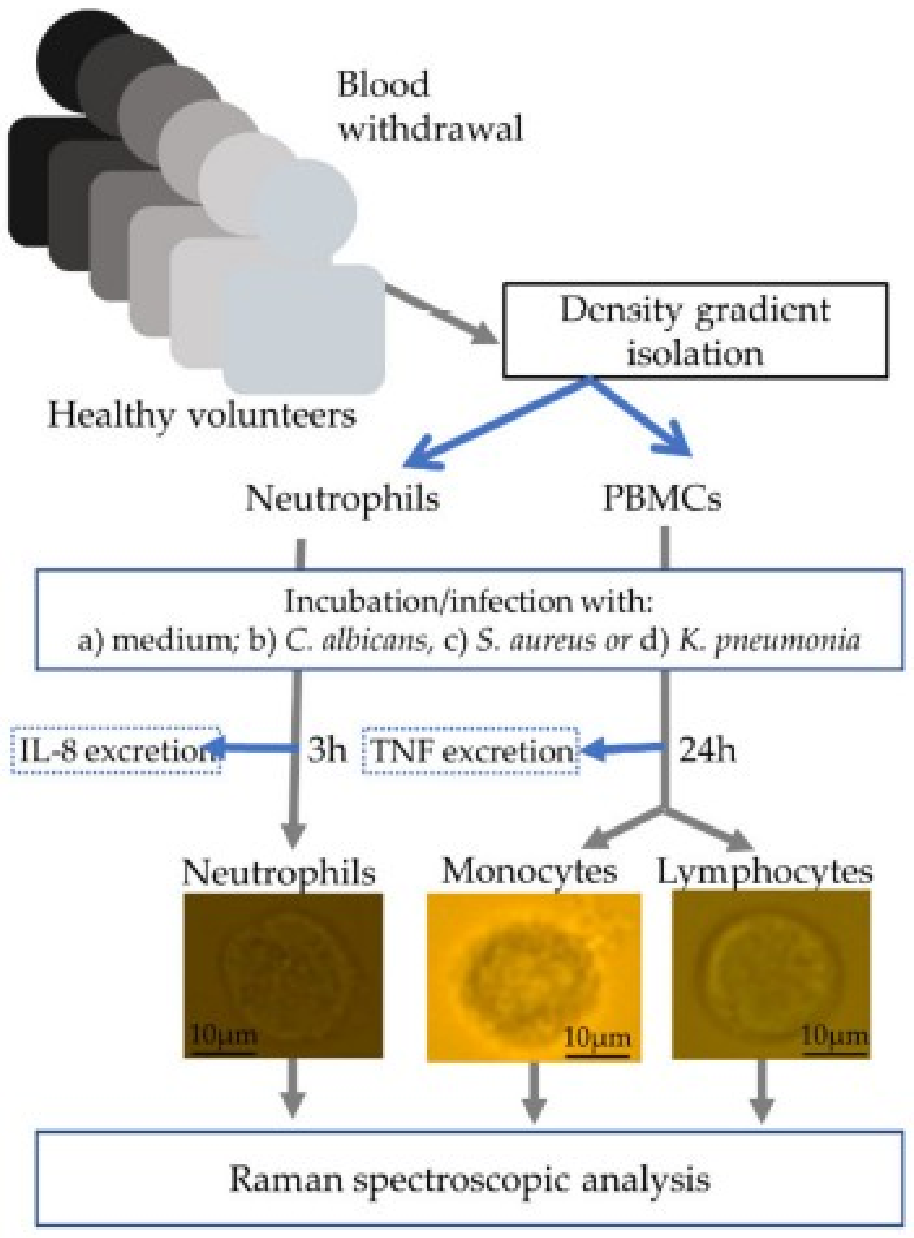,width=5cm} \\
          \end{tabular}
          \end{center}
         \caption{Representation of: cytokine model, cellular one, mixed cellular-cytokine model and Raman spectroscopic analysis}
        \label{fig c1}
          \end{figure}
      \subsection{Cellular model}
      The cellular model discussed in \cite{10tgq,33tgq} predicts the activation of the acquired immune response to S. aureus as a function of concentrations of lymphocytes, antibodies and plasma cells, denoted by: (T,B), F and A, respectively, together with activated and resting macrophages expressions, bacteria and antibodies spatially distributed in $1$ $cm^{3}$ of lump tissue $(x=(x_{1},x_{2},x_{3}))$ and represented by: M, A and F, respectively. The variation of concentrations of lymphocytes, antibodies and plasma cells is observed at the nearest of the lymph node. In this paper, we assume that the activated macrophages behave as antigens presenting cells and they move to the nearest lymph node where the specific response is triggered. So, these macrophages are obtained by interacting in space with the antigen in the tissue $(M_{A}(x,t))$ and at the lymph node $(M_{A}^{L}(t))$ by interacting with the lymphocyte expressions varying only in time. Thus, the dependent variables dealing with the cellular model are described as follows.
      \begin{itemize}
        \item Spatial variables $(pg/mm^{3})$: S. aureus pathogen ($y_{1}(x,t)$), resting macrophages ($y_{2}(x,t)$), activated macrophages ($y_{3}(x,t)$) and specific antibodies $(F(x,t))$.
        \item Temporal variables $(pg/mm^{3})$: plasma cells $(P(t))$, T-lymphocytes $(T(t))$, antibodies $(F^{l}(t))$, B-lymphocytes $(B(t))$ and average activated macrophages $(M_{l}^{a}(t))$.
      \end{itemize}

      \subsection{Coupled cellular-cytokine model}
      This work considers the dynamic of mixed cellular-cytokine model given in \cite{tgq} and described by a system of ordinary and partial differential equations. Assuming that the average tissue concentration obtained from the cellular model is the initial condition, the expected concentrations of TNF$\alpha$, IL6, IL8 and IL10 are functions of resting macrophages $y_{2}(x,t)$ and the activated ones $y_{3}(x,t)$. In some previous works, the authors developed this model to study the response of cytokines to LPS. Specifically, their findings have shown that a combination of cytokine model and the cell-based one is "well-defined" because LPS and pathogen-associated molecular patterns (PAMP) induce similar pro-inflammatory and anti-inflammatory responses \cite{38tgq}. The considered coupled model deals with the following unsteady variables: $y_{1}(x,t)$ indicates the bacteria, $y_{2}(x,t)$ and $y_{3}(x,t)$ denote resting and activated macrophages, respectively, while $y_{4}(t)$, $y_{5}(t)$, $y_{6}(t)$ and $y_{7}(t)$ represent: tumor necrosis factor alpha, interleukin 6, interleukin 8 and interleukin 10, respectively. We remind that tumor necrosis factor alpha, interleukin 6 and interleukin 8 are pro-inflammatory whereas interleukin 10 is an anti-inflammatory cytokine that plays an important role in the regulation of immune responses. Furthermore, interleukin 10 reduces the production of some pro-inflammatory cytokines (for example: interleukin 6 and tumor necrosis factor alpha).\\

      Let $\Omega$ be a bounded domain of $\mathbb{R}^{d}$, where $d=2$ or $3$, and $T$ be a positive real number. The initial-boundary value problem modeled by the S. aureus bacterium ($y_{1}(\cdot)$) is defined as
      \begin{equation}\label{1}
        \left\{
          \begin{array}{ll}
            \frac{\partial y_{1}}{\partial t}=\beta_{1}y_{1}(1-k_{1}^{-1}y_{1})-\mu_{1}y_{1}-\lambda_{2}y_{1}y_{2}-\lambda_{3}y_{1}y_{3}, & \hbox{on $\Omega\times[0,\text{\,}T]$} \\
            \text{\,}\\
            y_{1}(x,0)=\overline{y}_{1},\text{\,\,\,on\,\,\,}\overline{\Omega},\text{\,\,\,\,\,\,\,\,\,}\frac{\partial y_{1}}{\partial t}(x,t)=0, & \hbox{on $\partial\Omega\times[0,\text{\,}T]$.}
          \end{array}
        \right.
      \end{equation}
      Here: $\beta_{1}y_{1}(1-k_{1}^{-1}y_{1})$ denotes the logistic growth of the bacteria, -$\mu_{1}y_{1}$ indicates the natural decay rate of S. aureus without immune system processes through the natural decay coefficient $\mu_{1}$, the terms: -$\lambda_{2}y_{1}y_{2}$ and -$\lambda_{3}y_{1}y_{3}$ describe the phagocytosis of the pathogen S. aureus through resting macrophages and activated ones. the coefficients: $\beta_{1}$, $k_{1}$, $\mu_{1}$, $\lambda_{2}$ and $\lambda_{3}$, are positive constants which represent the carrying capacity, replication rate, natural decay and declined caused by resting and activated macrophages, respectively.\\

      Up-regulation and down-regulation for each cytokine are modeled by the use of sigmoidal functions defined as
      \begin{equation}\label{2}
      H_{c}^{u}(y)=\frac{y}{\eta_{xy}+y}\text{\,\,\,\,\,or\,\,\,\,\,}H_{c}^{d}(y)=\frac{\eta_{xy}}{\eta_{xy}+y},
      \end{equation}
      where $y$ is the cytokine inducing down-regulation (superscript: d) or up-regulation (superscript: u) of cytokine "c", $\eta$ denotes the half-maximum value. We add sigmoidal functions given by $(\ref{2})$ in all the equations dealing with cytokines to describe the relationship between cytokines (see Table 1 below).\\

      The mathematical model of resting macrophages ($y_{2}(\cdot)$) response to the bacterium S. aureus is given by
      \begin{equation}\label{3}
        \left\{
          \begin{array}{ll}
            \frac{\partial y_{2}}{\partial t}=\mu_{2}y_{2}(1-y_{2m}^{-1}y_{2})-\left[\gamma_{3}+k_{6}H_{c}^{u}(y_{4})H_{c}^{d}(y_{7})\right]y_{1}y_{2}, & \hbox{on $\Omega\times[0,\text{\,}T]$} \\
            \text{\,}\\
            y_{2}(x,0)=\overline{y}_{2}, & \hbox{on $\overline{\Omega}$}
          \end{array}
        \right.
      \end{equation}
       where $\mu_{2}y_{2}(1-y_{2m}^{-1}y_{2})$ represents the growth of the resting macrophages and -$\left[\gamma_{3}+k_{6}H_{c}^{u}(y_{4})H_{c}^{d}(y_{7})\right]y_{1}y_{2}$ designates macrophages activation at the rate $\gamma_{3}$ and $k_{6}$ in response to S. aureus and taking into account the cytokines $y_{4}(t)$ and $y_{7}(t)$ influences, respectively, $\mu_{2}$ is the influx rate associated with $y_{2}(x,t)$, $y_{2m}$ is the maximum of $y_{2}(x,t)$ on the domain $\Omega\times[0,\text{\,}T]$, $\overline{y}_{2}$ is the initial condition defined as the average of resting macrophages in the tissue. This average is an outcome of the cellular model simulation over $24h$.\\

       Activated macrophages ($y_{3}(\cdot)$) are modeled as follows
      \begin{equation}\label{4}
        \left\{
          \begin{array}{ll}
            \frac{\partial y_{3}}{\partial t}=-\mu_{3}y_{3}+\left[\gamma_{3}+k_{6}H_{c}^{u}(y_{4})H_{c}^{d}(y_{7})\right]y_{1}y_{2}, & \hbox{on $\Omega\times[0,\text{\,}T]$} \\
            \text{\,}\\
            y_{3}(x,0)=\overline{y}_{3}, & \hbox{on $\overline{\Omega}$.}
          \end{array}
        \right.
      \end{equation}
        In equation $(\ref{4})$, the term -$\mu_{3}y_{3}$ denotes the decay of $y_{3}(x,t)$ at rate $\mu_{3}$, $\left[\gamma_{3}+k_{6}H_{c}^{u}(y_{4})H_{c}^{d}(y_{7})\right]y_{1}y_{2}$ is the macrophages activation at the rates $\gamma_{3}$ and $k_{6}$ considering the influence of the cytokines $y_{4}(t)$ and $y_{7}(t)$. $\overline{y}_{3}$ is the initial condition, that is defined as the average of activated macrophages in the tissue. That is, an outcome of the cell-based model simulations during $24h$. Furthermore, the initial and boundary (of the tissue) concentrations of both $y_{2}(\cdot)$ and $y_{3}(\cdot)$ are constant.\\

        The dynamic of tumor necrosis factor alpha ($y_{4}(\cdot)$) is described as
      \begin{equation}\label{5}
        \left\{
          \begin{array}{ll}
            \frac{dy_{4}}{dt}=k_{7}H_{c}^{d}(y_{5})H_{c}^{d}(y_{7})y_{3}-k_{2}(y_{4}-q_{4}), & \hbox{on $\Omega\times[0,\text{\,}T]$} \\
            \text{\,}\\
            y_{4}(0)=\overline{y}_{4}:=0, & \hbox{}
          \end{array}
        \right.
      \end{equation}
        where, the first term: $k_{7}H_{c}^{d}(y_{5})H_{c}^{d}(y_{7})y_{3}$ denotes the down-regulation interactions that cytokines $y_{5}(\cdot)$ and $y_{7}(\cdot)$ have with the growth of $y_{4}(\cdot)$ (at rate $k_{7}$) mediated by the average concentrations of $y_{3}(\cdot)$, the second term: $-k_{2}(y_{4}-q_{4})$ represents the natural decay of $y_{4}(\cdot)$ at rate $k_{2}$. $q_{4}>0$, is a constant so called a resting level while $\overline{y}_{4}$ is the initial condition. Equation $(\ref{5})$ suggests that the rate of change of $y_{4}(\cdot)$ depends on $y_{3}(\cdot)$.\\

         The equation that models the evolution of interleukin 6 ($y_{5}(\cdot)$) is given by
      \begin{equation}\label{6}
        \left\{
          \begin{array}{ll}
            \frac{dy_{5}}{dt}=(k_{8}+k_{9}H_{c}^{u}(y_{4}))H_{c}^{d}(y_{5})H_{c}^{d}(y_{7})y_{3}-k_{3}(y_{5}-q_{5}), & \hbox{on $\Omega\times[0,\text{\,}T]$} \\
            \text{\,}\\
            y_{5}(0)=\overline{y}_{5}:=0, & \hbox{}
          \end{array}
        \right.
      \end{equation}
        where, the first term: $(k_{8}+k_{9}H_{c}^{u}(y_{4}))H_{c}^{d}(y_{5})H_{c}^{d}(y_{7})y_{3}$ designates the interaction between $y_{4}(\cdot)$ (up-regulation) and $y_{7}(\cdot)$ (down-regulation) affecting the $y_{5}(\cdot)$ production (at rate $k_{9}$) which also induces auto-negative feedback. The term: $-k_{3}(y_{5}-q_{5})$ indicates the natural decay (at rate $k_{3}$) of $y_{5}(\cdot)$ toward resting level: $q_{5}>0$.\\

        The time-dependent equation modeled by interleukin 8 ($y_{6}(\cdot)$) expressions is defined as
      \begin{equation}\label{7}
        \left\{
          \begin{array}{ll}
            \frac{dy_{6}}{dt}=(k_{10}+k_{11}H_{c}^{u}(y_{4}))H_{c}^{d}(y_{7})y_{3}-k_{4}(y_{6}-q_{6}), & \hbox{on $\Omega\times[0,\text{\,}T]$} \\
            \text{\,}\\
            y_{6}(0)=\overline{y}_{6}:=0. & \hbox{}
          \end{array}
        \right.
      \end{equation}
        Here, the term: $(k_{10}+k_{11}H_{c}^{u}(y_{4}))H_{c}^{d}(y_{7})y_{3}$ denotes the interaction between the opposing effects of $y_{4}(\cdot)$ (up-regulation) at rate $k_{11}$ and $y_{7}(\cdot)$ (down-regulation) at rate $k_{10}$, simulating growth of $y_{6}(\cdot)$ at a rate proportional to the average concentration of $y_{3}(\cdot)$ production whereas the second term: $-k_{4}(y_{6}-q_{6})$ represents the natural decay of $y_{6}(\cdot)$ at rate $k_{4}$, toward resting level: $q_{6}>0$.\\

        Finally, the ODE describing the dynamic of interleukin 10 ($y_{7}(\cdot)$) concentrations is given by
      \begin{equation}\label{8}
        \left\{
          \begin{array}{ll}
            \frac{dy_{7}}{dt}=(k_{12}+k_{13}H_{c}^{u}(y_{5}))y_{3}-k_{5}(y_{7}-q_{7}), & \hbox{on $\Omega\times[0,\text{\,}T]$} \\
            \text{\,}\\
            y_{7}(0)=\overline{y}_{7}:=0. & \hbox{}
          \end{array}
        \right.
      \end{equation}
       In this equation, the first term: $(k_{12}+k_{13}H_{c}^{u}(y_{5}))y_{3}$ is the up-regulation of $y_{7}(\cdot)$ due to $y_{5}(\cdot)$ (at rate $k_{13}$) and average concentration of $y_{3}(\cdot)$ (at rate $k_{12}$). The second term: $-k_{5}(y_{7}-q_{7})$ denotes the natural decay of $y_{7}(\cdot)$ at rate $k_{5}$. $q_{7}>0$ is the resting level.\\

       Moreover, the authors \cite{40tgq,36tgq} established that the migration of both S. aureus and macrophages should be observed as diffusion. The transfer rate of the cells from one side to another one is represented by the diffusion coefficients and it is proportional to the concentration gradient of the particles (cells). For the sake of convenience, the medium is assumed to be isotropic and it has the same diffusion coefficient in each direction \cite{36tgq}. That is, the diffusion term is modeled via quantity $D_{r}\Delta y_{s}$, where $r\in\{MA, MR\}$ and $s\in\{A,R\}$. $D_{r}$ denotes the diffusion coefficient of the macrophages in the tissue, while $\Delta$ refers to the Laplacian operator. Following the works discussed in \cite{10tgq,25tgq,33tgq} on the cell-based model, the average numbers of $y_{2}(\cdot)(t)$ and $y_{3}(\cdot)(t)$ are computed by integrating the resulting concentrations of each one over the domain $\Omega$. That is,
       \begin{equation}\label{9}
        \overline{y}_{l}(t)=\frac{1}{|\Omega|}\int_{\Omega}y_{l}(x,t)dx,
       \end{equation}
       where $l=2,3$, and $|\Omega|$ denotes the volume of $\Omega$.

       \subsection{Description of parameters}
       This subsection deals with the parameters included in the mixed cellular-cytokine model and described by the system of nonlinear ODEs-PDEs $(\ref{1})$-$(\ref{8})$. Since the human immune response modeling by these coupled equations seems to be too complex, for the sake simplicity, we assume in this paper that the cellular parameters, cytokine half-maximum value and hill function exponent parameters are constants.
       \begin{equation*}
       \begin{tabular}{|c|c|}
         \hline
         $D_{1}:$  bacteria diffusion coefficient & $D_{2}:$ RM diffusion coefficient  \\
         $D_{3}:$  AM diffusion coefficient & $\beta_{1}:$ replication rate of the bacteria \\
         $\lambda_{12}:$ destruction rate of opsonized bacteria by RM & $\mu_{1}:$ natural decay rate of the bacteria \\
         $\mu_{2}:$ natural decay rate of RM & $\mu_{3}:$ natural decay rate of AM\\
         $\gamma_{3}:$ rate of active RM & $\lambda_{2}:$ activation rate of macrophages\\
         $\lambda_{3}:$ destruction rate of bacteria by AM & $k_{1}:$ carrying capacity of the bacteria\\
         $\lambda_{13}:$ destruction rate of opsonized bacteria by AM & \\
         \hline
       \end{tabular}
       \end{equation*}
       where AM=$y_{3}(\cdot)$:= activated macrophages, RM=$y_{2}(\cdot)$:= resting macrophages and TNF$\alpha$=$y_{4}(\cdot)$:= tumor necrosis factor alpha. The above parameters determine the different rates of decay or growth for the bacteria and cells.
       \begin{equation*}
       \begin{tabular}{|c|c|}
         \hline
         $\alpha_{1}:$ migration rate of RM to site of infection & $k_{3}:$ activation rate of interleukin 6 (per hrs)  \\
         $\alpha_{2}:$ migration rate of AM to site of infection & $k_{4}:$ activation rate of interleukin 8 (per hrs) \\
         $k_{2}:$ activation rate of TNF$\alpha$ (per hrs) & $k_{5}:$ activation rate of interleukin 10 (per hrs) \\
         $k_{6}:$ activation rate of RM influenced by $y_{4}(\cdot)$ & $k_{7}:$ upregulation of $y_{4}(\cdot)$ by AM \\
         $k_{8}:$ upregulation of $y_{5}(\cdot)$ by AM & $k_{9}:$ upregulation of $y_{5}(\cdot)$ by TNF$\alpha$ \\
         $k_{10}:$ upregulation of $y_{6}(\cdot)$ by AM & $k_{11}:$ upregulation of $y_{6}(\cdot)$ by TNF$\alpha$ \\
         $k_{12}:$ upregulation of $y_{7}(\cdot)$ by AM & $k_{13}:$ upregulation of $y_{7}(\cdot)$ by interleukin 6 \\
         \hline
       \end{tabular}
       \end{equation*}
       We recall that the constants $k_{sl}$ denote the rate of change in the up-regulation rate of the cytokine secreted from activated macrophages while $k_{s}$ represent the activation or elimination of cytokine. The initial values of these parameters are based on the predicted conditions of the model activated with low dose of the bacterium S. aureus.
       \begin{equation*}
       \begin{tabular}{|c|c|}
         \hline
         $q_{4}:$ concentration of TNF$\alpha$ in absence of pathogem & $q_{5}:$ concentration of IL6 in absence of antigen  \\
         $q_{6}:$ concentration of IL8 in absence of bacteria & $q_{7}:$ concentration of IL10 in absence of microbes \\
         \hline
       \end{tabular}
       \end{equation*}
       where IL6=$y_{5}(\cdot)$:= interleukin 6, IL8=$y_{6}(\cdot)$:= interleukin 8 and IL10=$y_{7}(\cdot)$:= interleukin 10. it worth mentioning that the source terms $q_{l}$ are also used to set the initial conditions for each cytokine. The values of these parameters are obtained from the initial predicted conditions of the model in the absence of the bacterium S. aureus. The half-maximum value parameter $\eta_{sl}$ describes the effective cytokine concentration at which targeted cytokine activity should reach half-maximum with units of $pgmL^{-1}$.\\

       \textbf{Table 1}
       \begin{equation*}
       \begin{tabular}{|c|c|}
         \hline
         $\eta_{45}:$ HMV assoc. down-reg of TNF$\alpha$ by IL6 & $H_{y_{7}}^{u}(y_{5}):$ HFE assoc. up-reg of IL10 by IL6 \\
         $\eta_{57}:$ HMV assoc. down-reg of IL6 by IL10 & $H_{y_{5}}(y_{5}):$ HFE assoc. auto-neg feedback of IL6\\
         $\eta_{54}:$ HMV assoc. up-reg of IL6 by TNF$\alpha$ & $H_{y_{6}}^{u}(y_{4}):$ HFE assoc. up-reg of IL8 by TNF$\alpha$\\
         $\eta_{67}:$ HMV assoc. down-reg of IL8 by IL10 &  $H_{y_{4}}^{d}(y_{7}):$ HFE assoc. down-reg of TNF$\alpha$ by IL10\\
         $\eta_{47}:$ HMV assoc. down-reg of TNF$\alpha$ by IL10  & $H_{y_{5}}^{u}(y_{4}):$ HFE assoc. up-reg of IL6 by TNF$\alpha$ \\
         $\eta_{55}:$ HMV assoc. auto-neg feedback of IL6  & $H_{y_{5}}^{d}(y_{7}):$ HFE assoc. down-reg of IL6 by IL10 \\
         $\eta_{64}:$ HMV assoc. up-reg of IL8 by TNF$\alpha$ & $H_{y_{6}}^{d}(y_{7}):$ HFE assoc. down-reg of IL8 by IL10  \\
         $\eta_{75}:$ HMV assoc. up-reg of IL10 by IL6 & $H_{y_{4}}^{d}(y_{5}):$ HFE assoc. down-reg of TNF$\alpha$ by IL6  \\
         \hline
       \end{tabular}
       \end{equation*}
       where HMV:=half-max value, assoc. down-reg:= associated down-regulation, HFE:=hill function exponent, assoc. up-reg:= associated with up-regulation, assoc. auto-neg:=associated with auto-negative.\\

       In the literature \cite{28tgq,36tgq}, the authors discussed the values of the above parameters and the obtained results are provided in the following Table 2.\\

       \textbf{Table 2}
       \begin{equation*}
          \begin{tabular}{|c|c|c|c|c|c|}
         \hline
         parameter & value & unit & parameter & value & unit \\
         \hline
         $D_{1}$ & $3.7\times10^{-15}$ & $mm^{3}/day$ & $\beta_{1}$ & $2.0$ & $day^{-1}$ \\
         $D_{2}$ & $4.32\times10^{-2}$ & $mm^{3}/day$ & $k_{1}$ & $5\times10^{-2}$ & $cell/mm^{3}$\\
         $D_{3}$ & $3\times10^{-1}$ & $mm^{3}/day$ & $\mu_{1}$ & $0.1$ & $day^{-1}$\\
         $\mu_{2}$ & $3.3\times10^{-2}$ & $day^{-1}$ & $q_{4}$ & $0.14$ & relative concentration \\
         $\mu_{3}$ & $7\times10^{-2}$ & $day^{-1}$ & $q_{5}$ & $0.6$ & relative concentration \\
         $\gamma_{3}$ & $8.2\times10^{-2}$ & $mm^{3}/cell.day$ & $q_{6}$ & $0.2$ & relative concentration  \\
         $\lambda_{2}$ & $5.98\times10^{-3}$ & $mm^{3}/cell.day$ & $q_{7}$ & $0.15$ & relative concentration \\
         $\lambda_{3}$ & $5.98\times10^{-2}$ & $mm^{3}/cell.day$ & $\eta_{45}$ & $560$ & relative concentration \\
         $\lambda_{13}$ & $7.14\times10^{-2}$ & $mm^{6}/cell^{2}.day$ & $\eta_{47}$ & $17.4$ & relative concentration \\
         $\lambda_{12}$ & $1.66\times10^{-3}$ & $mm^{6}/cell^{2}.day$ & $\eta_{57}$ & $34.8$ & relative concentration \\
         $\alpha_{1}$ & $4\times10^{0}$ & $day^{-1}$ & $\eta_{55}$ & $560$ & relative concentration \\
         $\alpha_{2}$ & $10^{-3}$ & $day^{-1}$ & $\eta_{54}$ & $185$ & relative concentration \\
         $k_{2}$ & $2\times10^{-1}$ & $day^{-1}$ & $\eta_{64}$ & $185$ & relative concentration \\
         $k_{3}$ & $4.64$ & $day^{-1}$ & $\eta_{67}$ & $17.4$ & relative concentration \\
         $k_{4}$ & $0.464$ & $day^{-1}$ & $\eta_{75}$ & $560$ & relative concentration \\
         $k_{5}$ & $1.1$ & $day^{-1}$ & $H_{y_{4}}^{d}(y_{7})$ & $3.0$ & dimensionless \\
         $k_{6}$ & $8.65$ & $hr^{-1}$ & $H_{y_{4}}^{d}(y_{5})$ & $2.0$ & dimensionless \\
         $k_{7}$ & $1.5$ & $\frac{rel.cyt.conc.}{day\text{\,}of\text{\,}cell}$ & $H_{y_{5}}(y_{5})$ & $1.0$ & dimensionless \\
         $k_{8}$ & $10^{-2}$ & $\frac{rel.cyt.conc.}{day\text{\,}of\text{\,}cell}$ & $H_{y_{7}}^{u}(y_{5})$ & $3.68$ & dimensionless \\
         $k_{9}$ & $8.1\times10^{-1}$ & $\frac{rel.cyt.conc.}{day\text{\,}of\text{\,}cell}$ & $H_{y_{5}}^{u}(y_{4})$ & $2.0$ & dimensionless \\
         $k_{10}$ & $5.6\times10^{-2}$ & $\frac{rel.cyt.conc.}{day\text{\,}of\text{\,}cell}$ & $H_{y_{5}}^{d}(y_{7})$ & $4.0$ & dimensionless \\
         $k_{11}$ & $5.6\times10^{-1}$ & $\frac{rel.cyt.conc.}{day\text{\,}of\text{\,}cell}$ & $H_{y_{6}}^{u}(y_{4})$ & $3.0$ & dimensionless \\
         $k_{12}$ & $1.9\times10^{-1}$ & $\frac{rel.cyt.conc.}{day\text{\,}of\text{\,}cell}$ & $H_{y_{6}}^{d}(y_{7})$ & $1.5$ & dimensionless \\
         $k_{13}$ & $1.91\times10^{-2}$ & $\frac{rel.cyt.conc.}{day\text{\,}of\text{\,}cell}$ &  &  &  \\
         \hline
       \end{tabular}
       \end{equation*}
        where rel. cyt. conc.:=relative cytokine concentration.

        \subsection{Statistical approach}
         A regression model can be used to compare the numerical results and the experimental data. The analysis should consider computed relative concentrations of the pathogen S. aureus $y_{1}(\cdot)$, resting macrophages $y_{2}(\cdot)$, activated macrophages $y_{3}(\cdot)$, tumor necrosis factor alpha $y_{4}(\cdot)$, interleukin 6 $y_{5}(\cdot)$, interleukin 8 $y_{6}(\cdot)$ and interleukin 10 $y_{7}(\cdot)$, together with the experimental concentrations for cellular and cytokines induced by either gram-positive bacteria, gram-negative bacteria, lipoteichoic acid or peptidoglycan. Furthermore, for either cellular or cytokine the regression model should be linear:
        \begin{equation*}
            z=ay+b,
        \end{equation*}
        where $a$ denotes the slope of each linear least square regression and $b$ is the $z$-intercept.

        \section{Development of the second-order explicit predictor-corrector numerical scheme}\label{sec3}
        In this section, we develop a new second-order explicit predictor-corrector numerical technique in a computed solution of the cellular-cytokine model described by the nonlinear unsteady equations $(\ref{1})$-$(\ref{8})$.\\

       Let $N$ be a positive integer and $\sigma:=\Delta t=\frac{T}{N}$, be the step size. Set $t_{k}=k\sigma$,  $t_{k+\frac{1}{2}}=\frac{t_{k}+t_{k+1}}{2}$, for $k=0,1,2,...,N$, and $\mathcal{T}_{\sigma}=\{t_{k},\text{\,\,}k=0,1,...,N\}$, be a regular partition of $[0,\text{\,}T]$. For the convenience of writing, we set $y_{r}(t):=y_{r}(x,t)$, for $r=4,...,7$, $y_{j}(x,t_{k})=y_{j}^{k}(x)$, for $(x,t_{k})\in\Omega\times\mathcal{T}_{\sigma}$, $j=1,2,...,7$, and $y^{k}(x)=[y_{1}^{k}(x),...,y_{7}^{k}(x)]\in\mathbb{R}^{7}$. The space of mesh functions defined on $\Omega\times\mathcal{T}_{\sigma}$ is given by $\mathcal{F}_{\sigma}=\{y^{k}(x),\text{\,\,}0\leq k\leq N,\text{\,}x\in\Omega\}$. We introduce the vector functions $y$ and $F$ defined from $\Omega\times[0,\text{\,}T]$ to $\mathbb{R}^{7}$ and $\mathbb{R}^{7}\times\Omega\times[0,\text{\,}T]$ to $\mathbb{R}^{7}$, respectively, by
       \begin{equation}\label{10}
        y(x,t)=[y_{1}(x,t),...,y_{7}(x,t)] \text{\,\,\,\,\,and\,\,\,\,\,} F(y(x,t))=[F_{1}(y(x,t)),...,F_{7}(y(x,t))],
       \end{equation}
       where
       \begin{equation}\label{12}
        \begin{tabular}{cc}
         $F_{1}(y(x,t)) = \beta_{1}(1-k_{1}^{-1}y_{1})y_{1}-\mu_{1}y_{1}-\lambda_{2}y_{1}y_{2}-\lambda_{3}y_{1}y_{3}$, \\
         \text{\,}\\
         $F_{2}(y(x,t)) = \mu_{2}(1-y_{2m}^{-1}y_{2})y_{2}-[\gamma_{3}+k_{6}H_{c}^{u}(y_{4})H_{c}^{d}(y_{7})]y_{1}y_{2}$, \\
         \text{\,}\\
         $F_{3}(y(x,t)) = [\gamma_{3}+k_{6}H_{c}^{u}(y_{4})H_{c}^{d}(y_{7})]y_{1}y_{2}-\mu_{3}y_{3}$, \\
         \text{\,}\\
         $F_{4}(y(x,t)) = k_{7}H_{c}^{d}(y_{5})H_{c}^{d}(y_{7})y_{3}-k_{2}(y_{4}-q_{4})$, \\
         \text{\,}\\
         $F_{5}(y(x,t)) = [k_{8}+k_{9}H_{c}^{u}(y_{4})]H_{c}^{d}(y_{5})H_{c}^{d}(y_{7})]y_{3}-k_{3}(y_{5}-q_{5})$, \\
         \text{\,}\\
         $F_{6}(y(x,t)) = [k_{10}+k_{11}H_{c}^{u}(y_{4})]H_{c}^{d}(y_{7})y_{3}-k_{4}(y_{6}-q_{6})$, \\
         \text{\,}\\
         $F_{7}(y(x,t)) = [k_{12}+k_{13}H_{c}^{u}(y_{5})]y_{3}-k_{5}(y_{7}-q_{7})$.
        \end{tabular}
       \end{equation}
       Utilizing equations $(\ref{10})$-$(\ref{12})$, it is not hard to observe that the system of nonlinear differential equations $(\ref{1})$-$(\ref{8})$ can be expressed as
       \begin{equation}\label{13}
        \frac{\partial y}{\partial t}(x,t)=F(y(x,t)),
       \end{equation}
       where $\frac{\partial y}{\partial t}=[\frac{\partial y_{1}}{\partial t},...,\frac{\partial y_{7}}{\partial t}]$. Subjects to initial-boundary conditions
       \begin{equation}\label{14}
      y_{j}(x,0)=\overline{y}_{j},\text{\,\,\,for\,\,\,}j=1,...,7,\text{\,\,\,on\,\,\,\,}\overline{\Omega},
      \end{equation}
       \begin{equation}\label{15}
      \frac{\partial y_{1}}{\partial t}(x,t)=0,\text{\,\,\,on\,\,\,\,}\partial\Omega\times[0,\text{\,}T].
      \end{equation}
      The integration of $(\ref{13})$ over the interval $(t_{k},\text{\,}t_{k+\frac{1}{2}})$ provides
      \begin{equation*}
        y(x,t_{k+\frac{1}{2}})=y(x,t_{k})+\int_{t_{k}}^{t_{k+\frac{1}{2}}}F(y(x,t))dt,
      \end{equation*}
      which is equivalent to
      \begin{equation}\label{16}
      y^{k+\frac{1}{2}}(x)=y^{k}(x)+\int_{t_{k}}^{t_{k+\frac{1}{2}}}F(y(x,t))dt.
      \end{equation}
      Let $P_{j}^{(1)}(x,t)$ be the first-order polynomial approximating the functions $F_{j}(y(x,t))$ at the mesh points $(t_{m},F_{j}(y^{m}(x)))$, for $m\in\{k,k+\frac{1}{2}\}$, thus
      \begin{equation}\label{17}
      F_{j}(y(x,t))=P_{j}^{(1)}(x,t)+E_{j}^{(1)}(y(x,t)),
      \end{equation}
      where
      \begin{equation*}
      P_{j}^{(1)}(x,t)=\frac{t-t_{k+\frac{1}{2}}}{t_{k}-t_{k+\frac{1}{2}}}F_{j}(y^{k}(x))+\frac{t-t_{k}}{t_{k+\frac{1}{2}}-t_{k}}F_{j}(y^{k+\frac{1}{2}}(x))=
       \frac{\sigma}{2}[F_{j}(y^{k+\frac{1}{2}}(x))-F_{j}(y^{k}(x))]+
      \end{equation*}
      \begin{equation}\label{18}
      t_{k+\frac{1}{2}}F_{j}(y^{k}(x))-t_{k}F_{j}(y^{k+\frac{1}{2}}(x)),
      \end{equation}
      and the associated error $E_{j}$ is defined as
      \begin{equation}\label{19}
      E_{j}^{(1)}(y(x,t))=\frac{1}{2}(t-t_{k})(t-t_{k+\frac{1}{2}})\frac{\partial^{2}F_{j}}{\partial t^{2}}(y(x,t_{\epsilon}^{1})),
      \end{equation}
      where $t_{\epsilon}^{1}$ is between the minimum and maximum of $t_{k+\frac{1}{2}}$, $t_{k}$ and $t$.\\

      Integrating both sides of equation $(\ref{18})$ over the interval $[t_{k},\text{\,}t_{k+\frac{1}{2}}]$, results in
      \begin{equation}\label{20}
      \int_{t_{k}}^{t_{k+\frac{1}{2}}}P_{j}^{(1)}(x,t)dt=\frac{1}{\sigma}[F_{j}(y^{k+\frac{1}{2}}(x))-F_{j}(y^{k}(x))](t_{k+\frac{1}{2}}^{2}-t_{k}^{2})+
      \frac{2}{\sigma}[t_{k+\frac{1}{2}}F_{j}(y^{k}(x))-t_{k}F_{j}(y^{k+\frac{1}{2}}(x))](t_{k+\frac{1}{2}}-t_{k}).
      \end{equation}
       Since $t_{m}=m\sigma$, so $(t_{k+\frac{1}{2}}^{2}-t_{k}^{2})=(t_{k+\frac{1}{2}}-t_{k})(t_{k+\frac{1}{2}}+t_{k})=(2k+\frac{1}{2})\frac{\sigma^{2}}{2}$ and
       $[t_{k+\frac{1}{2}}F_{j}(y^{k}(x))-t_{k}F_{j}(y^{k+\frac{1}{2}}(x))](t_{k+\frac{1}{2}}-t_{k})=\frac{\sigma^{2}}{2}[(k+\frac{1}{2})F_{j}(y^{k}(x))-
       kF_{j}(y^{k+\frac{1}{2}}(x))]$. These facts along with equation $(\ref{20})$ give
       \begin{equation}\label{21}
      \int_{t_{k}}^{t_{k+\frac{1}{2}}}P_{j}^{(1)}(x,t)dt=\frac{\sigma}{4}[F_{j}(y^{k}(x))+F_{j}(y^{k+\frac{1}{2}}(x))].
      \end{equation}
       In addition, the integration of equation $(\ref{19})$ on the interval $[t_{k},\text{\,}t_{k+\frac{1}{2}}]$, yields
      \begin{equation*}
      \int_{t_{k}}^{t_{k+\frac{1}{2}}}E_{j}^{(1)}(y(x,t))dt=\frac{1}{2}\int_{t_{k}}^{t_{k+\frac{1}{2}}}(t-t_{k})(t-t_{k+\frac{1}{2}})\frac{\partial^{2}F_{j}}{\partial t^{2}}(y(x,t))dt.
      \end{equation*}
       The absolute value in both sides of this equation implies
       \begin{equation}\label{22}
      \left|\int_{t_{k}}^{t_{k+\frac{1}{2}}}E_{j}^{(1)}(y(x,t))dt\right|\leq\frac{1}{2}\int_{t_{k}}^{t_{k+\frac{1}{2}}}|t-t_{k}||t-t_{k+\frac{1}{2}}|
      \left|\frac{\partial^{2}F_{j}}{\partial t^{2}}(y(x,t))\right|dt\leq\frac{\sigma^{3}}{16}\underset{0\leq t\leq T}{\sup}
      \left|\frac{\partial^{2}F_{j}}{\partial t^{2}}(y(x,t))\right|.
      \end{equation}
       But, utilizing equations $(\ref{1})$-$(\ref{8})$, it follows that the functions $y_{j}(\cdot)$ have continuous first-order partial derivatives on the bounded domain $\overline{\Omega}\times[0,\text{\,}T]$. This fact together with equations $(\ref{2})$ and $(\ref{12})$ show that the functions $F_{j}$ have continuous second-order partial derivatives. Thus, $\underset{0\leq t\leq T}{\sup}\left|\frac{\partial^{2}F_{j}}{\partial t^{2}}(y(x,t))\right|\leq C_{0}$, where $C_{0}>0$, is a constant independent of $\sigma$. Using this and estimate $(\ref{22})$, one can write
       \begin{equation}\label{23}
       \int_{t_{k}}^{t_{k+\frac{1}{2}}}E_{j}^{(1)}(y(x,t))dt=O(\sigma^{3}),\text{\,\,\,for\,\,\,\,}j=1,2,...,7.
      \end{equation}
       Integrating equation $(\ref{17})$ on $[t_{k},\text{\,}t_{k+\frac{1}{2}}]$ and utilizing equations $(\ref{21})$ and $(\ref{23})$, to get
       \begin{equation*}
      \int_{t_{k}}^{t_{k+\frac{1}{2}}}F_{j}(y(x,t))dt=\frac{\sigma}{4}[F_{j}(y^{k}(x))+F_{j}(y^{k+\frac{1}{2}}(x))]+O(\sigma^{3}),\text{\,\,\,for\,\,\,\,}j=1,2,...,7.
      \end{equation*}
       Substituting this into equation $(\ref{16})$, to obtain
       \begin{equation}\label{24}
       y^{k+\frac{1}{2}}(x)=y^{k}(x)+\frac{\sigma}{4}[F(y^{k}(x))+F(y^{k+\frac{1}{2}}(x))]+\mathcal{O}(\sigma^{3}),
      \end{equation}
       where $\mathcal{O}(k^{3})=(O(\sigma^{3}),...,O(\sigma^{3}))$. We should approximate the term $F_{j}(y^{k}(x))+F_{j}(y^{k+\frac{1}{2}}(x))$, by the sum $c_{1}F_{j}(y^{k}(x))+c_{2}F_{j}[y^{k}(x)+\sigma pF(y^{k}(x))]$, in which the scalars $c_{1}$, $c_{2}$ and $p$ are chosen so that the following equation holds
       \begin{equation*}
        \frac{2}{\sigma}\left(y^{k+\frac{1}{2}}(x)-y^{k}(x)\right)-\frac{1}{2}[F(y^{k}(x))+F(y^{k+\frac{1}{2}}(x))]=\mathcal{O}(\sigma^{2}).
       \end{equation*}
       Expanding the Taylor series for $y_{j}(x,\cdot)$ and $F_{j}$ about the points $t_{k}$ and $y^{k}(x)$, respectively, with steplength $\frac{\sigma}{2}$ utilizing forward difference formulation, simple computations result in
       \begin{equation*}
        \frac{2}{\sigma}\left(y_{j}^{k+\frac{1}{2}}(x)-y_{j}^{k}(x)\right)-\frac{1}{2}[F_{j}(y^{k}(x))+F_{j}(y^{k+\frac{1}{2}}(x))]=(1-\frac{c_{1}+c_{2}}{2})F_{j}(y^{k}(x))+
        \frac{\sigma}{4}(1-2c_{2}p)\frac{\partial F_{j}}{\partial t}(y^{k}(x))+O(\sigma^{2}).
       \end{equation*}
       The right side of this equations equals $O(\sigma^{2})$, if and only if, $c_{1}+c_{2}=2$ and $c_{2}p=\frac{1}{2}$. For instance, we can take $c_{1}=\frac{3}{2}$, $c_{2}=\frac{1}{2}$ and $p=1$. So,
       \begin{equation*}
        F_{j}(y^{k}(x))+F_{j}(y^{k+\frac{1}{2}}(x))=\frac{3}{2}F_{j}(y^{k}(x))+\frac{1}{2}F_{j}[y^{k}(x)+\sigma F(y^{k}(x))],\text{\,\,\,for\,\,\,}j=1,2,...,7.
       \end{equation*}
       This is equivalent to
       \begin{equation}\label{24a}
        F(y^{k}(x))+F(y^{k+\frac{1}{2}}(x))=\frac{3}{2}F(y^{k}(x))+\frac{1}{2}F[y^{k}(x)+\sigma F(y^{k}(x))].
       \end{equation}
       Plugging equations $(\ref{24})$ and $(\ref{24a})$, we obtain
       \begin{equation*}
       y^{k+\frac{1}{2}}(x)=y^{k}(x)+\frac{\sigma}{8}\left[3F(y^{k}(x))+F\left(y^{k}(x)+\sigma F(y^{k}(x))\right)\right]+\mathcal{O}(\sigma^{3}).
      \end{equation*}
       Truncating the error term $\mathcal{O}(\sigma^{3})$ and replacing the analytical solution $y(x,\cdot)$ with the computed one $Y(x,\cdot)$, this gives
       \begin{equation}\label{25}
       Y^{k+\frac{1}{2}}(x)=Y^{k}(x)+\frac{\sigma}{8}\left[3F(Y^{k}(x))+F\left(Y^{k}(x)+\sigma F(Y^{k}(x))\right)\right].
      \end{equation}
       Equation $(\ref{25})$ represents the predictor phase of the desired algorithm.\\

       In a similar manner, denoting by $P_{j}^{(2)}(x,t)$ be the linear polynomial in $t$ interpolating the function $F_{j}(y(x,t))$ at the mesh points $(t_{k+\frac{1}{2}},F_{j}(y^{k+\frac{1}{2}}(x)))$ and $(t_{k+1},F_{j}(y^{k+1}(x)))$, and $E_{j}^{(2)}(y(x,t))$ be the corresponding error. Replacing $F(y(x,t))$ with $P^{(2)}(x,t)+E^{(2)}(y(x,t))$ into equation $(\ref{13})$ and integration over the interval $[\text{\,}t_{k+\frac{1}{2}},t_{k+1}]$, where $P^{(2)}(x,t)=[P_{1}^{(2)}(x,t),...,P_{7}^{(2)}(x,t)]$ and $E^{(2)}(y(x,t))=[E_{1}^{(2)}(y(x,t)),...,E_{7}^{(2)}(y(x,t))]$, straightforward  computations provide

       \begin{equation}\label{26}
       y^{k+1}(x)=y^{k+\frac{1}{2}}(x)+\frac{\sigma}{4}[F(y^{k+\frac{1}{2}}(x))+F(y^{k+1}(x))]+\mathcal{O}(\sigma^{3}).
      \end{equation}
       By the application of Taylor series, the quantity $F(y^{k+\frac{1}{2}}(x))+F(y^{k+1}(x))$ can be approximated as
       \begin{equation*}
        F_{j}(y^{k+\frac{1}{2}}(x))+F_{j}(y^{k+1}(x))=F_{j}(y^{k+\frac{1}{2}}(x))+F_{j}[y^{k+\frac{1}{2}}(x)+\frac{\sigma}{2}F(y^{k+\frac{1}{2}}(x))],
        \text{\,\,\,for\,\,\,}j=1,2,...,7.
       \end{equation*}
       Thus,
       \begin{equation}\label{27}
       F(y^{k+\frac{1}{2}}(x))+F(y^{k+1}(x))=F(y^{k+\frac{1}{2}}(x))+F[y^{k+\frac{1}{2}}(x)+\frac{\sigma}{2}F(y^{k+\frac{1}{2}}(x))].
       \end{equation}
       Combining equations $(\ref{26})$ and $(\ref{27})$, this yields
       \begin{equation}\label{28}
       y^{k+1}(x)=y^{k+\frac{1}{2}}(x)+\frac{\sigma}{4}\left[F(y^{k+\frac{1}{2}}(x))+F\left(y^{k+\frac{1}{2}}(x)+\frac{\sigma}{2} F(y^{k+\frac{1}{2}}(x))\right)\right]+\mathcal{O}(\sigma^{3}).
      \end{equation}
       Tracking the infinitesimal term $\mathcal{O}(\sigma^{3})$ and replacing the exact solution $y(x,\cdot)$ with the approximate one $Y(x,\cdot)$, to get the corrector step of the proposed approach
       \begin{equation}\label{29}
       Y^{k+1}(x)=Y^{k+\frac{1}{2}}(x)+\frac{\sigma}{4}\left[F(Y^{k+\frac{1}{2}}(x))+F\left(Y^{k+\frac{1}{2}}(x)+\frac{\sigma}{2}F(Y^{k+\frac{1}{2}}(x))\right)\right].
      \end{equation}
       But equation $(\ref{15})$ suggests that $\frac{\partial y_{1}}{\partial t}(x,t)=0$, on $\partial\Omega\times[0,T]$. This indicates that $y_{1}(x,t)$ is constant on $\partial\Omega\times[0,T]$. Since, $y_{1}(x,0)=\overline{y}_{1}$, so $y_{1}^{k}(x)=\overline{y}_{1}$, for $k=0,1,...,N$, $x\in\partial\Omega$. Plugging this equation together with equations $(\ref{14})$, $(\ref{25})$ and $(\ref{29})$, to obtain the developed second-order explicit predictor-corrector numerical technique for solving the mathematical model of the dynamic of coupled cellular-cytokine problem $(\ref{1})$-$(\ref{8})$. That is, for $k=0,1,...,N-1$, and every $x\in\Omega$,
       \begin{equation}\label{s1}
       Y^{k+\frac{1}{2}}(x)=Y^{k}(x)+\frac{\sigma}{8}\left[3F(Y^{k}(x))+F\left(Y^{k}(x)+\sigma F(Y^{k}(x))\right)\right],
      \end{equation}
      \begin{equation}\label{s2}
       Y^{k+1}(x)=Y^{k+\frac{1}{2}}(x)+\frac{\sigma}{4}\left[F(Y^{k+\frac{1}{2}}(x))+F\left(Y^{k+\frac{1}{2}}(x)+\frac{\sigma}{2}F(Y^{k+\frac{1}{2}}(x))\right)\right],
      \end{equation}
      with initial and boundary conditions
       \begin{equation}\label{s3}
       Y_{j}^{0}=\overline{y}_{j},\text{\,\,\,for\,\,\,}j=1,2,...,7,\text{\,\,\,on\,\,\,}\Omega,\text{\,\,\,\,and\,\,\,\,}
        Y^{k}_{1}(x)=\overline{y}_{1},\text{\,\,\,for\,\,\,}k=0,1,2,...,N,\text{\,\,\,}x\in\partial\Omega.
       \end{equation}
       We remind that $Y=[Y_{1},...,Y_{7}]$ and $F=[F_{1},...,F_{7}]$.

       \begin{theorem}\label{t}
       The numerical approach $(\ref{s1})$-$(\ref{s2})$ is zero-stable and second-order accurate for any values of the initial-boundary conditions given by $(\ref{s3})$.
       \end{theorem}

       \begin{proof}
        A combination of approximations $(\ref{s1})$ and$(\ref{s2})$ results in
        \begin{equation*}
       Y^{k+1}(x)-Y^{k}(x)=\frac{\sigma}{8}\left[3F(Y^{k}(x))+F\left(Y^{k}(x)+\sigma F(Y^{k}(x))\right)\right]+
      \end{equation*}
       \begin{equation*}
       \frac{\sigma}{4}\left[F(Y^{k+\frac{1}{2}}(x))+F\left(Y^{k+\frac{1}{2}}(x)+\frac{\sigma}{2}F(Y^{k+\frac{1}{2}}(x))\right)\right].
      \end{equation*}
        It follows from this equation that the first characteristic polynomial in term of $\lambda^{\frac{1}{2}}$ of the developed explicit predictor-corrector numerical scheme is defined as
       \begin{equation*}
        P_{2}(\lambda^{\frac{1}{2}})=(\lambda^{\frac{1}{2}})^{2}-1=(\lambda^{\frac{1}{2}}-1)(\lambda^{\frac{1}{1}}+1).
       \end{equation*}
       It is not difficult to see that the roots of this polynomial are: $\lambda_{1}^{\frac{1}{2}}=1$ and $\lambda_{2}^{\frac{1}{2}}=-1$. Since the two roots are simple and lie in the closed unit disc, it comes from the definition of zero-stability \cite{dahl1956,dahl1963} that the proposed scheme $(\ref{s1})$-$(\ref{s3})$ is zero-stable. Furthermore, since the truncation errors in both approximations $(\ref{25})$ and $(\ref{29})$ equal $\mathcal{O}(\sigma^{3})$, the proof of the second-order accuracy for the constructed technique $(\ref{s1})$-$(\ref{s3})$ is similar to that established in \cite{12en}.
       \end{proof}

         \section{Numerical experiments and Discussions}\label{sec4}
         This section presents some numerical simulations to demonstrate the utility and effectiveness of the new second-order explicit predictor-corrector scheme $(\ref{s1})$-$(\ref{s3})$ in approximate solutions of the mathematical model on the dynamic of cytokine levels and human immune cell activation in response to the pathogen S. aureus $(\ref{1})$-$(\ref{8})$. Since the proposed technique is an explicit predictor-corrector approach and second-order accuracy, it is faster and more efficient than a wide set of statistical and numerical methods discussed in the literature for solving general systems of mixed ODEs/PDEs modeling real-world problems \cite{17krbf,18krbf,20krbf} and references therein. Some data used in the simulations are taken from \cite{tgq}. Both tables and graphs (Tables $3$-$4$ and Figures $\ref{fig1}$-$\ref{fig2}$) suggest that the approximate solutions (macrophages and cytokine expressions) increase until a maximum concentration, so called "peaks", after approximately $t$ hours, and become constant when the antigens are eliminated by the immune response. The activated macrophages behave as the pathogens presenting cells and they move to the nearest lymph node where the specific response is triggered. The peaks of TNF$\alpha$, IL6, IL8 and IL10 (see Tables $3$-$4$), agree with the model results (Figure $\ref{fig2}$). Figure $\ref{fig1}$ considers the relationship between S. aureus $(Y_{1}(\cdot))$, resting macrophages $(Y_{2}(\cdot))$ and activated macrophages $(Y_{3}(\cdot))$ whereas Figures $\ref{fig2}$ indicates how the changes in activated macrophages and resting macrophages impact TNF$\alpha$ $(Y_{4}(\cdot))$, IL6 $(Y_{5}(\cdot))$, IL8 $(Y_{6}(\cdot))$ and IL10 $(Y_{7}(\cdot))$. Additionally, the simulated results do not include some factors such as: humoral immune response effects, neutrophil flux or complement response. The initial concentrations of the pathogens, macrophages and cytokines are chosen to match well those reported in the ex vivo studies \cite{12tgq,35tgq}. Furthermore, the numerical analysis shows that the constructed predictor-corrector explicit method $(\ref{s1})$-$(\ref{s3})$ accurately predicts the dynamic of human immune cell activation and cytokine levels in response to S. aureus. Thus, the new algorithm can be considered as a robust tool to predict ex vivo and in vivo experimental data induced by a given antigen (Tables $3$-$4$). The graphs (Figures $\ref{fig1}$-$\ref{fig2}$) also indicate that the body response varies based on the damage inflicted by the bacteria S. aureus. Because of the activation of cytokines in the human immune response, the numerical experiments are performed assuming low initial concentration of gram-positive bacteria S. aureus as discussed in \cite{tgq}. However, high expressions of the pathogen S. aureus due to the tissue damage associated with endotoxicity of gram-positive bacteria provoke a fast increase in cellular and cytokine responses (see Figure $\ref{fig3}$). This case is not analyzed in this work and should be considered as the topic of our future investigations. Furthermore, high concentrations of the antigens are not sufficiently discussed in the literature because of the lack of data to validate the increased expressions \cite{tgq}.\\
         \text{\,}\\
         \text{\,}\\

         \textbf{Table 3.} Concentrations of S. aureus, macrophages and cytokines (in cells$/mm^{3}$) at point $x=(2\times10^{-3},\text{\,}10^{-3},\text{\,}2\times10^{-3})$.
         \begin{equation*}
         \small{\begin{tabular}{|c|c|c|c|c|c|c|c|}
           \hline
           time (hrs) & S. aureus  & rest. macroph & act. macroph & TNF$\alpha$ & IL6 & IL8 & IL10 \\
            \hline
           0 & $2.0\times10^{-1}$  & $10^{-2}$ & $5.0\times10^{-3}$ & $0.0\times10^{0}$ & $0.0\times10^{0}$ & $0.0\times10^{0}$ & $0.0\times10^{0}$ \\
           1 & $2.22\times10^{-2}$ & $1.29\times10^{-2}$ & $4.80\times10^{-3}$ & $2.47\times10^{-2}$ & $5.93\times10^{-1}$ & $7.20\times10^{-2}$ & $9.84\times10^{-2}$ \\
           3 & $1.91\times10^{-2}$ & $2.18\times10^{-2}$ & $4.40\times10^{-3}$ & $6.33\times10^{-2}$ & $6.00\times10^{-1}$ & $1.48\times10^{-1}$ & $1.45\times10^{-1}$ \\
           6 & $1.90\times10^{-2}$ & $4.19\times10^{-2}$ & $4.20\times10^{-3}$ & $9.88\times10^{-2}$ & $6.00\times10^{-1}$ & $1.88\times10^{-1}$ & $1.505\times10^{-1}$ \\
           9 & $1.90\times10^{-2}$ & $6.42\times10^{-2}$ & $4.40\times10^{-3}$ & $1.18\times10^{-1}$ & $6.00\times10^{-1}$ & $1.97\times10^{-1}$ & $1.508\times10^{-1}$ \\
           12& $1.90\times10^{-2}$ & $8.09\times10^{-2}$ & $5.00\times10^{-3}$ & $1.29\times10^{-1}$ & $6.00\times10^{-1}$ & $1.97\times10^{-1}$ & $1.509\times10^{-1}$ \\
           15& $1.90\times10^{-2}$ & $8.98\times10^{-2}$ & $5.70\times10^{-3}$ & $1.35\times10^{-1}$ & $6.00\times10^{-1}$ & $2.00\times10^{-1}$ & $1.510\times10^{-1}$ \\
           18& $1.90\times10^{-2}$ & $9.39\times10^{-2}$ & $6.40\times10^{-3}$ & $1.38\times10^{-1}$ & $6.00\times10^{-1}$ & $2.004\times10^{-1}$ & $1.511\times10^{-1}$ \\
           21& $1.90\times10^{-2}$ & $9.55\times10^{-2}$ & $7.00\times10^{-3}$ & $1.40\times10^{-1}$ & $6.00\times10^{-1}$ & $2.005\times10^{-1}$ & $1.512\times10^{-1}$ \\
           24& $1.90\times10^{-2}$ & $9.61\times10^{-2}$ & $7.50\times10^{-3}$ & $1.41\times10^{-1}$ & $6.00\times10^{-1}$ & $2.005\times10^{-1}$ & $1.513\times10^{-1}$ \\
           peak & --- & $9.61\times10^{-2}$ & $7.50\times10^{-3}$ & $1.411\times10^{-1}$ & $6.00\times10^{-1}$ & $2.005\times10^{-1}$ & $1.513\times10^{-1}$ \\
           \hline
         \end{tabular}}
         \end{equation*}
         \text{\,}\\
         \text{\,}\\
         \textbf{Table 4.} Concentrations of S. aureus, macrophages and cytokines (in cells$/mm^{3}$) at point $x=(2\times10^{-3},\text{\,}3\times10^{-3},\text{\,}6\times10^{-3})$
         \begin{equation*}
         \small{\begin{tabular}{|c|c|c|c|c|c|c|c|}
           \hline
           time (hrs) & S. aureus  & rest. macroph & act. macroph & TNF$\alpha$ & IL6 & IL8 & IL10 \\
            \hline
           0 & $4.68\times10^{-1}$  & $2.33\times10^{-2}$ & $1.17\times10^{-2}$ & $0.0\times10^{0}$ & $0.0\times10^{0}$ & $0.0\times10^{0}$ & $0.0\times10^{0}$ \\
           1 & $2.24\times10^{-2}$ & $2.87\times10^{-2}$ & $1.13\times10^{-2}$ & $2.51\times10^{-2}$ & $5.93\times10^{-1}$ & $7.22\times10^{-2}$ & $9.92\times10^{-2}$ \\
           3 & $1.91\times10^{-2}$ & $4.29\times10^{-2}$ & $1.03\times10^{-2}$ & $6.43\times10^{-2}$ & $6.00\times10^{-1}$ & $1.50\times10^{-1}$ & $1.46\times10^{-1}$ \\
           6 & $1.90\times10^{-2}$ & $6.54\times10^{-2}$ & $9.40\times10^{-3}$ & $1.00\times10^{-1}$ & $6.00\times10^{-1}$ & $1.88\times10^{-1}$ & $1.515\times10^{-1}$ \\
           9 & $1.90\times10^{-2}$ & $8.15\times10^{-2}$ & $9.10\times10^{-3}$ & $1.20\times10^{-1}$ & $6.00\times10^{-1}$ & $1.98\times10^{-1}$ & $1.516\times10^{-1}$ \\
           12& $1.90\times10^{-2}$ & $9.02\times10^{-2}$ & $9.00\times10^{-3}$ & $1.30\times10^{-1}$ & $6.00\times10^{-1}$ & $2.00\times10^{-1}$ & $1.516\times10^{-1}$ \\
           15& $1.90\times10^{-2}$ & $9.40\times10^{-2}$ & $9.10\times10^{-3}$ & $1.36\times10^{-1}$ & $6.00\times10^{-1}$ & $2.01\times10^{-1}$ & $1.516\times10^{-1}$ \\
           18& $1.90\times10^{-2}$ & $9.55\times10^{-2}$ & $9.20\times10^{-3}$ & $1.39\times10^{-1}$ & $6.00\times10^{-1}$ & $2.01\times10^{-1}$ & $1.516\times10^{-1}$ \\
           21& $1.90\times10^{-2}$ & $9.62\times10^{-2}$ & $9.30\times10^{-3}$ & $1.41\times10^{-1}$ & $6.00\times10^{-1}$ & $2.01\times10^{-1}$ & $1.516\times10^{-1}$ \\
           24& $1.90\times10^{-2}$ & $9.64\times10^{-2}$ & $9.40\times10^{-3}$ & $1.42\times10^{-1}$ & $6.00\times10^{-1}$ & $2.01\times10^{-1}$ & $1.516\times10^{-1}$ \\
           peak & --- & $9.64\times10^{-2}$ & $1.17\times10^{-2}$ & $1.419\times10^{-1}$ & $6.00\times10^{-1}$ & $2.01\times10^{-1}$ & $1.516\times10^{-1}$ \\
           \hline
         \end{tabular}}
         \end{equation*}

         \section{General conclusion and future works}\label{sec5}
         In this paper, we have developed a second-order explicit predictor-corrector numerical approach in simulated results of a coupled cellular-cytokine model defined by the system of nonlinear equations $(\ref{1})$-$(\ref{8})$. The theory has suggested that the proposed numerical technique is second-order convergent and stable for any value of the initial datum (Theorem $\ref{t}$). Additionally, the study indicates that the new algorithm is faster and more efficient than a broad range of statistical techniques and numerical schemes discussed in the literature for solving such systems of nonlinear differential equations \cite{7en,18krbf,17krbf,20krbf,29tgq}. Furthermore, the graphs show that the cellular model predicts the relationship between the S. aureus bacteria ($Y_{1}$), resting macrophages ($Y_{2}$) and activated ones ($Y_{3}$) whereas the cytokine model analyzes how the changes in the resting and activated macrophages influence the propagation of tumor necrosis factor alpha ($Y_{4}$), interleukin 6 ($Y_{5}$), interleukin 8 ($Y_{6}$) and interleukin 10 ($Y_{7}$). The developed numerical approach is used to investigating and predicting the dynamic of cytokine levels and human immune cell activation in response to the pathogen S. aureus. This suggests that the proposed technique $(\ref{s1})$-$(\ref{s3})$ can be observed as a robust tool to predicting in vivo and ex vivo experimental data induced by a specific bacterium (see Tables $3$-$4$). The considered bacteria neutrophil interactions are human-specific and may impact the predicted cytokine average expressions. More accurate simulated results should be obtained by incorporating the effect of neutrophils in the ODE modeling the interleukin 8. Finally, the model defined by the initial-boundary value problem $(\ref{1})$-$(\ref{8})$ does not incorporate relevant complement proteins. Complement response of human immune system to gram-positive bacterium (for instance, S. aureus) is fundamental in the activation of chemoattractants for phagocytosis of the antigens. Specifically, cytokine and complement responses have overlapping biological effects on the human body under septic conditions. However, solve a coupling of ordinary and partial differential equations by fast and efficient numerical approaches provides a more realistic representation of the complex relationships within the immune system and may serve as a good test of drugs in silico. Our future works will develop a third-order explicit numerical technique to predicting the dynamic of mixed cellular-cytokine model with complement proteins.
       \text{\,}\\
        \subsection*{Ethical Approval}
     Not applicable.
     \subsection*{Availability of supporting data}
     Not applicable.
     \subsection*{Declaration of Interest Statement}
     The authors declare that they have no conflict of interests.
     \subsection*{Funding}
     Not applicable.
     \subsection*{Authors' contributions}
       The whole work has been carried out by the three authors.

            \newpage

         \begin{figure}
         \begin{center}
          Approximate solutions at different spatial points $x$.
          \begin{tabular}{c c}
         \psfig{file=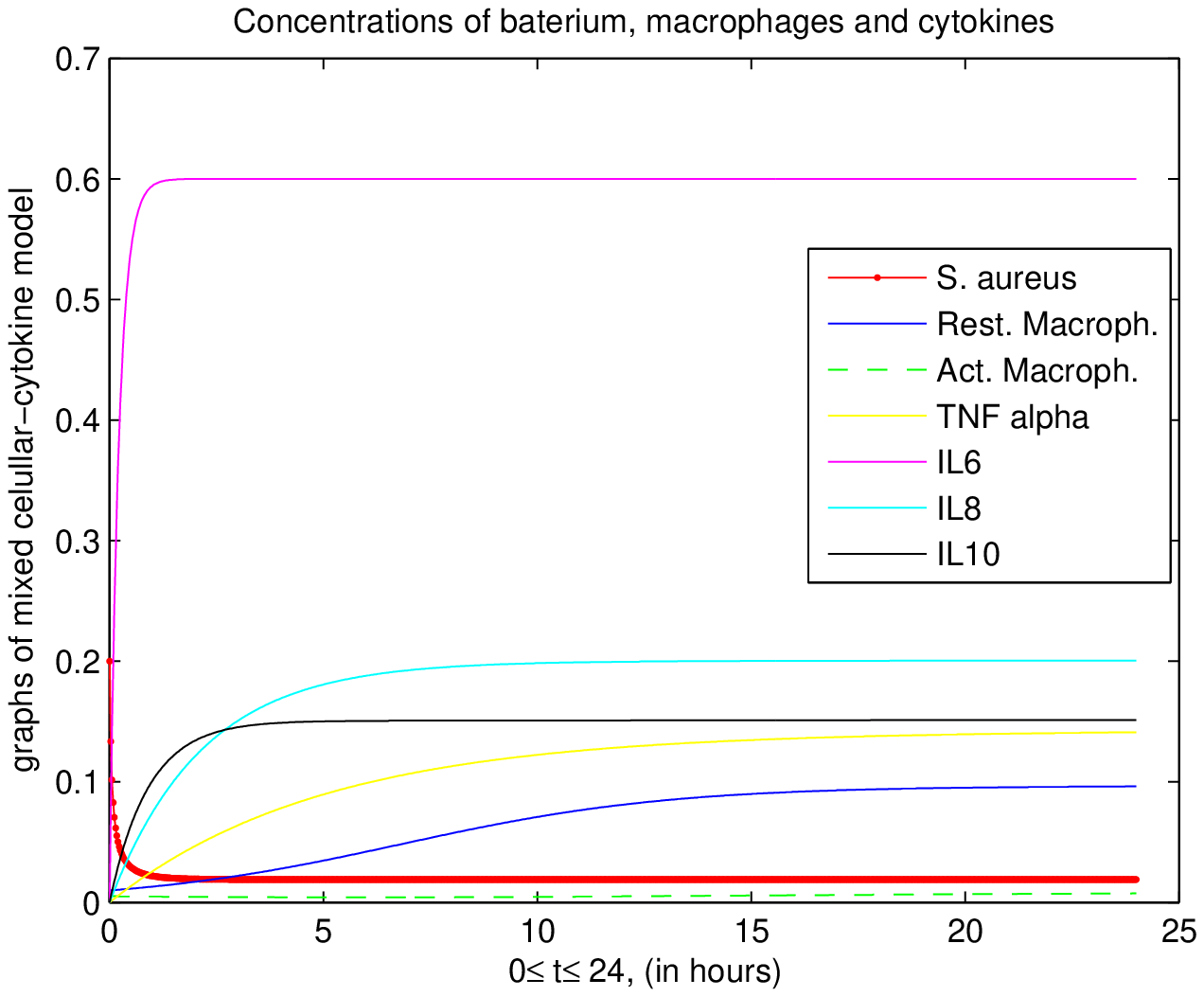,width=6cm}  & \psfig{file=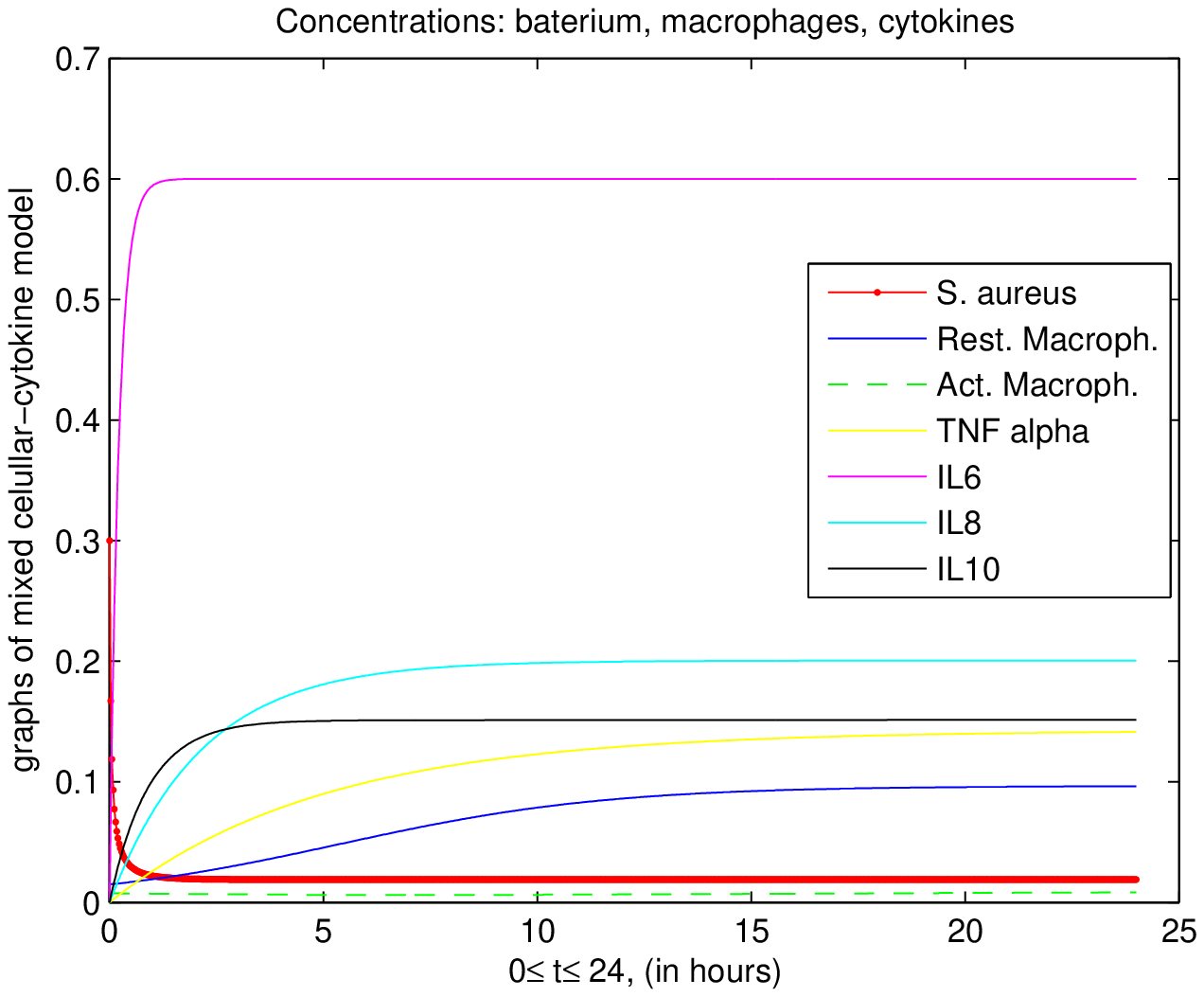,width=6cm}\\
         $x=(2\times10^{-3},\text{\,}10^{-3},\text{\,}2\times10^{-3})$  & $x=(3\times10^{-3},\text{\,}1.5\times10^{-3},\text{\,}3\times10^{-3})$\\
         \text{\,}\\
         \psfig{file=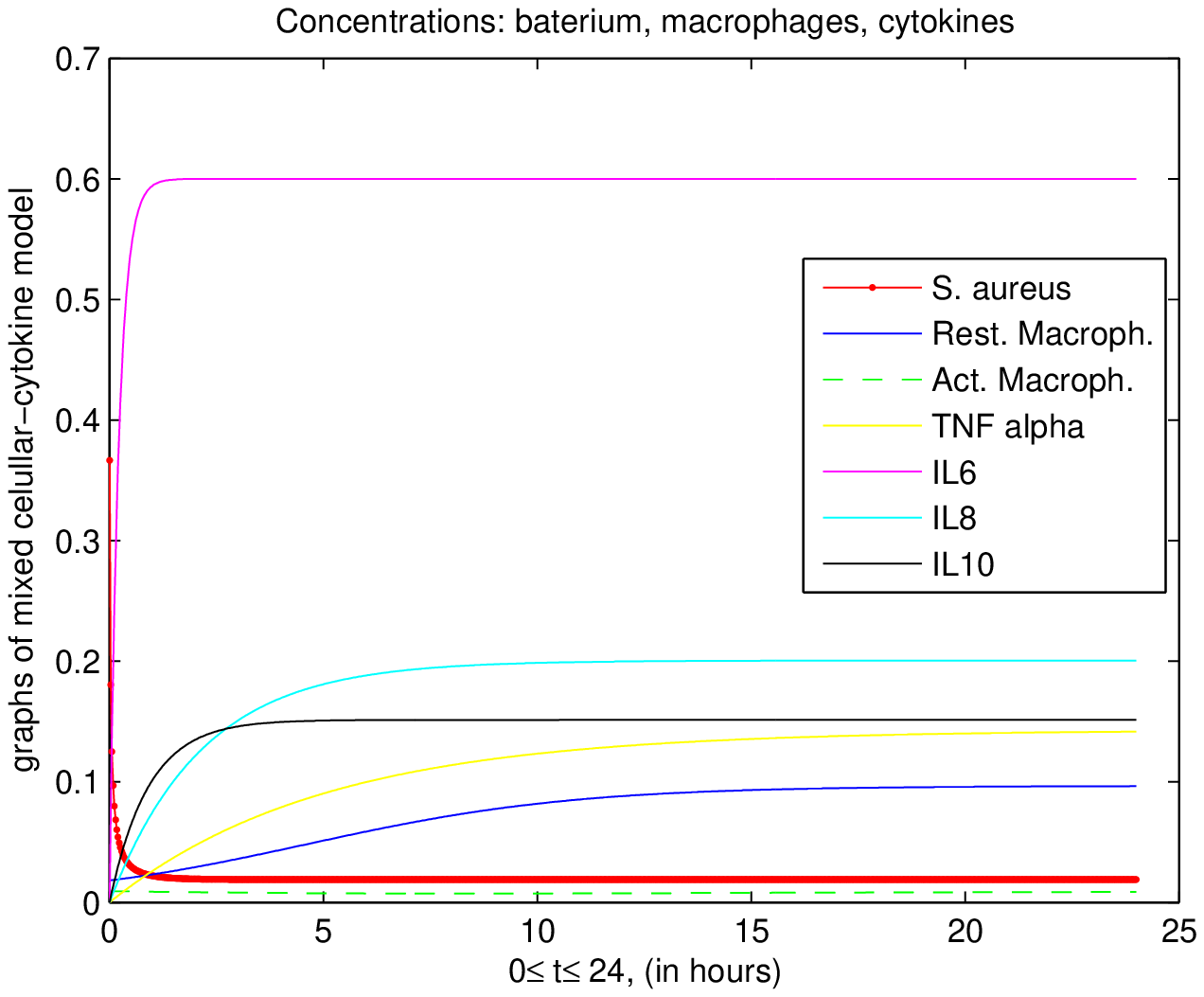,width=6cm}  & \psfig{file=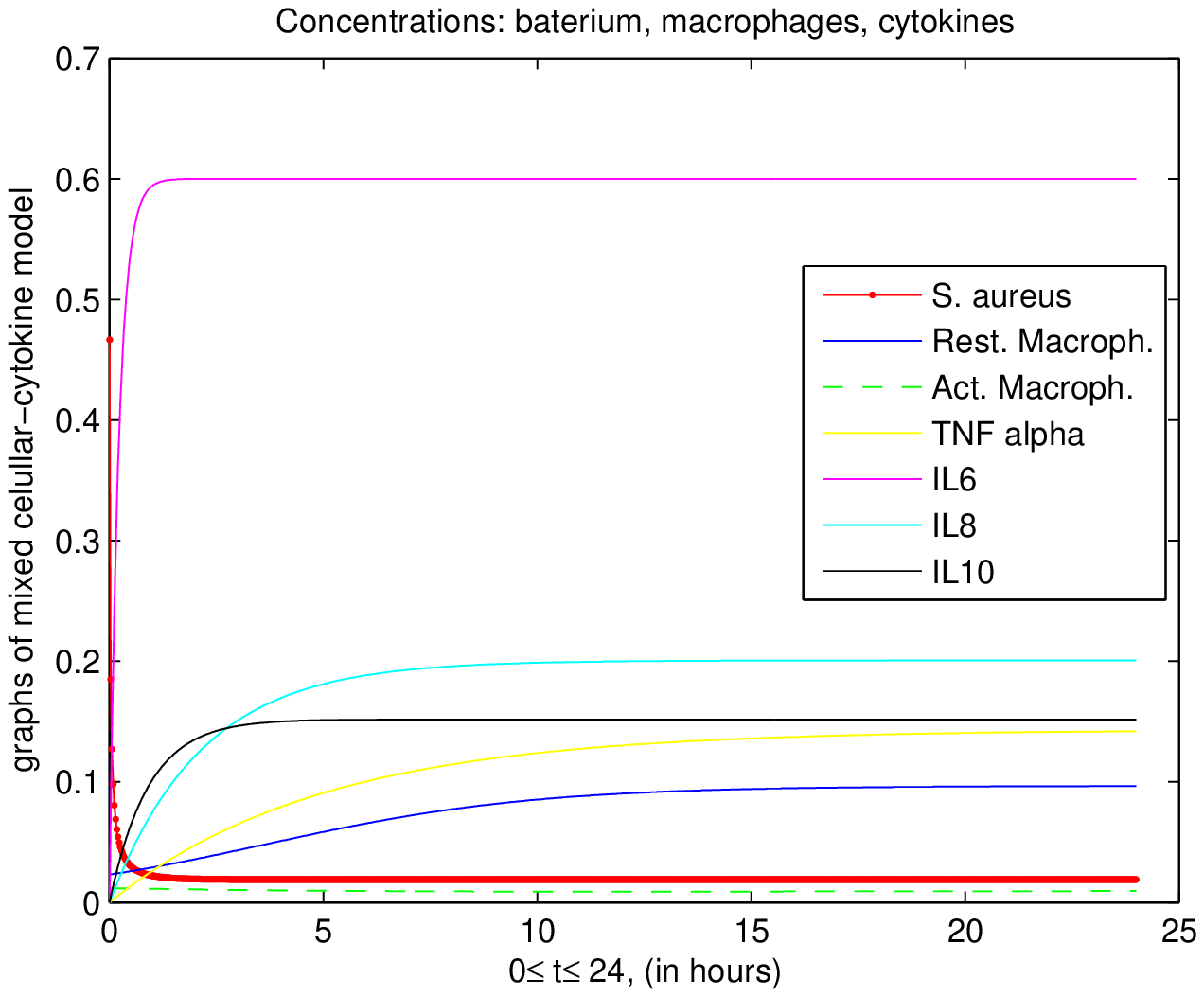,width=6cm}\\
         $x=(2\times10^{-2},\text{\,}5\times10^{-3},\text{\,}10^{-3})$  & $x=(2\times10^{-3},\text{\,}3\times10^{-3},\text{\,}6\times10^{-3})$\\
         \end{tabular}
        \end{center}
         \caption{ Concentrations: S. aureus, resting macroph., activated macroph., TNF alpha, IL6, IL8 and IL10}
        \label{fig1}
          \end{figure}

         \begin{figure}
         \begin{center}
         Simulated results at point $x=(2\times10^{-3},\text{\,}10^{-3},\text{\,}2\times10^{-3})$
          \begin{tabular}{c c}
         \psfig{file=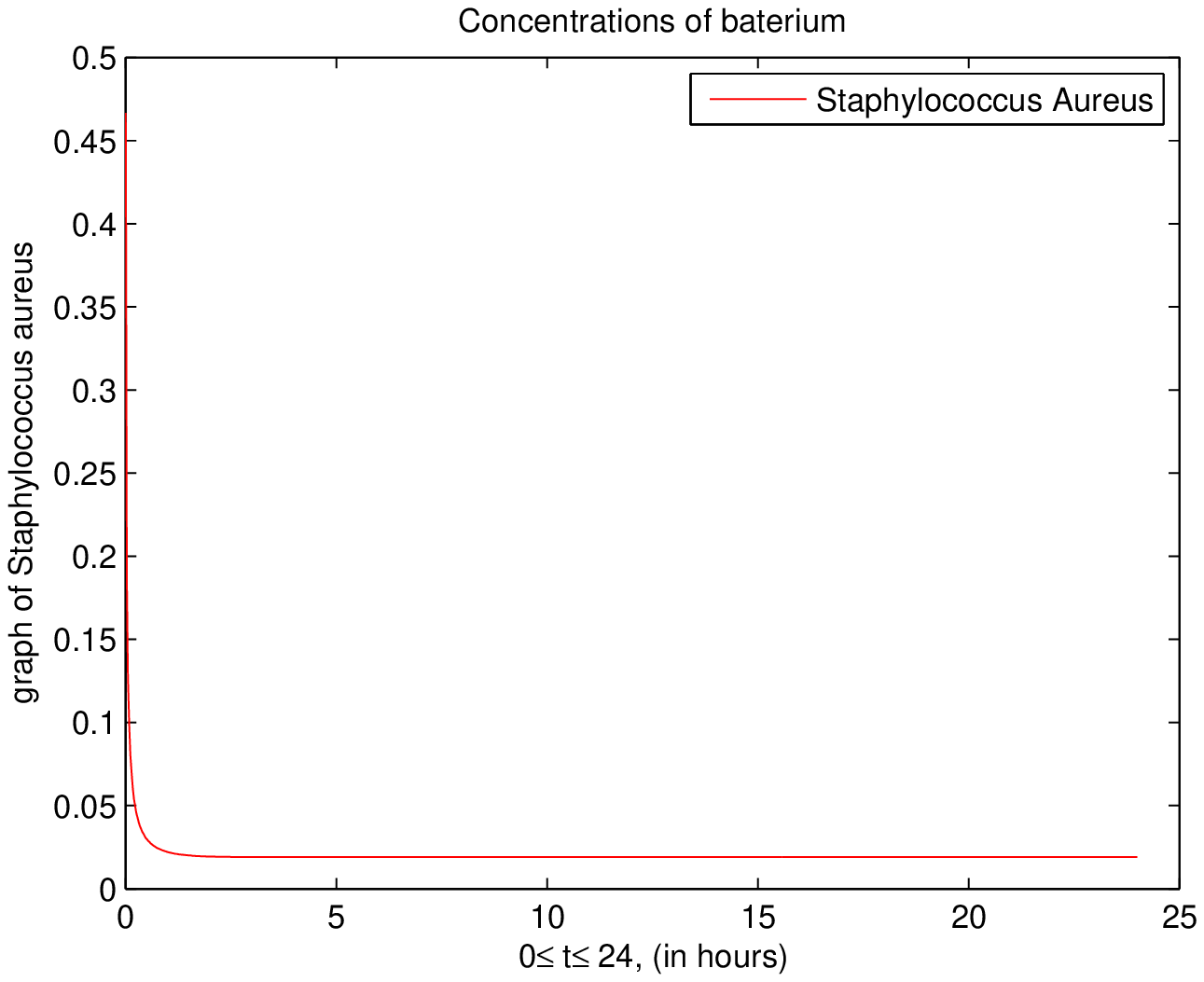,width=6cm}  & \psfig{file=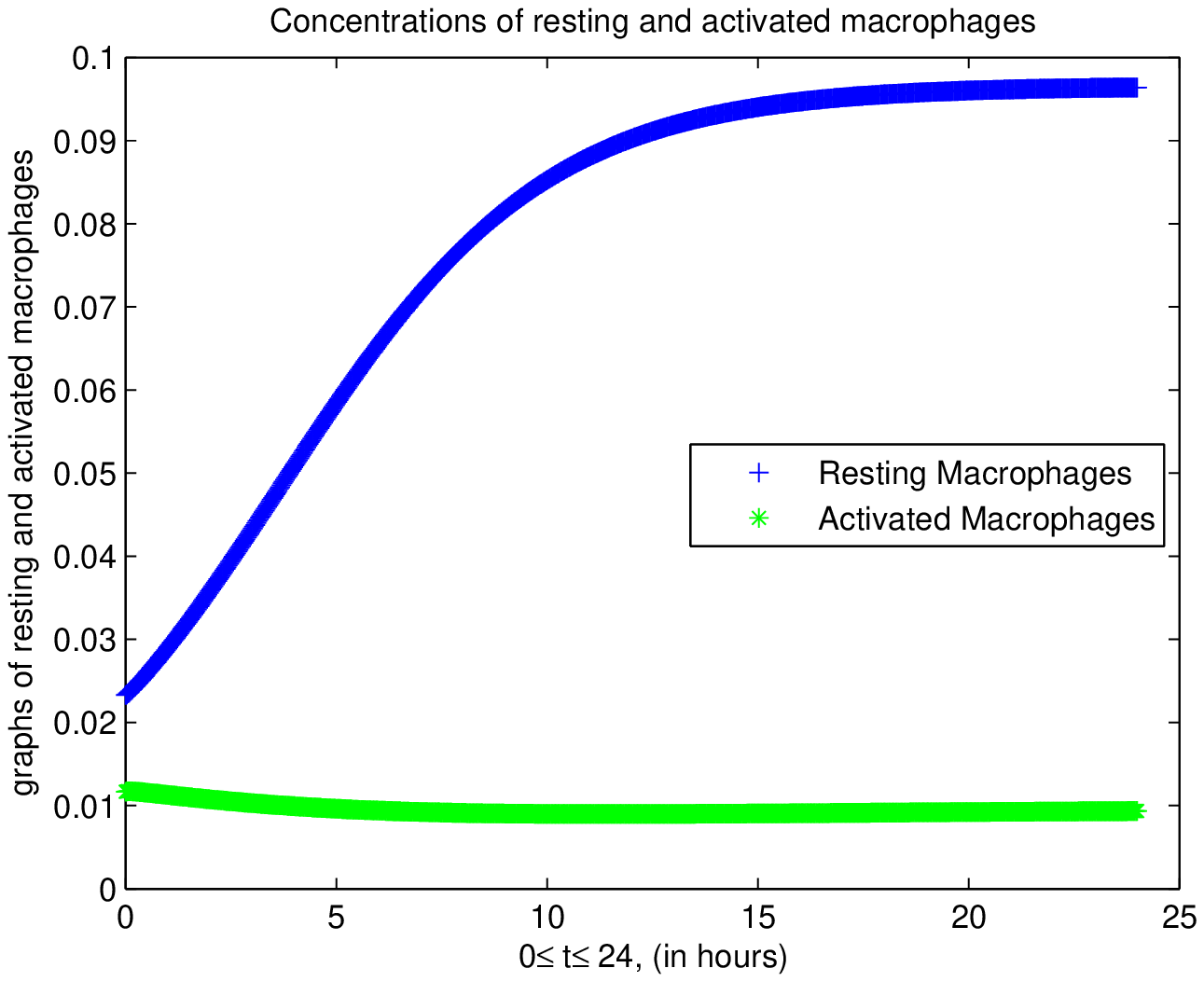,width=6cm}\\
         \psfig{file=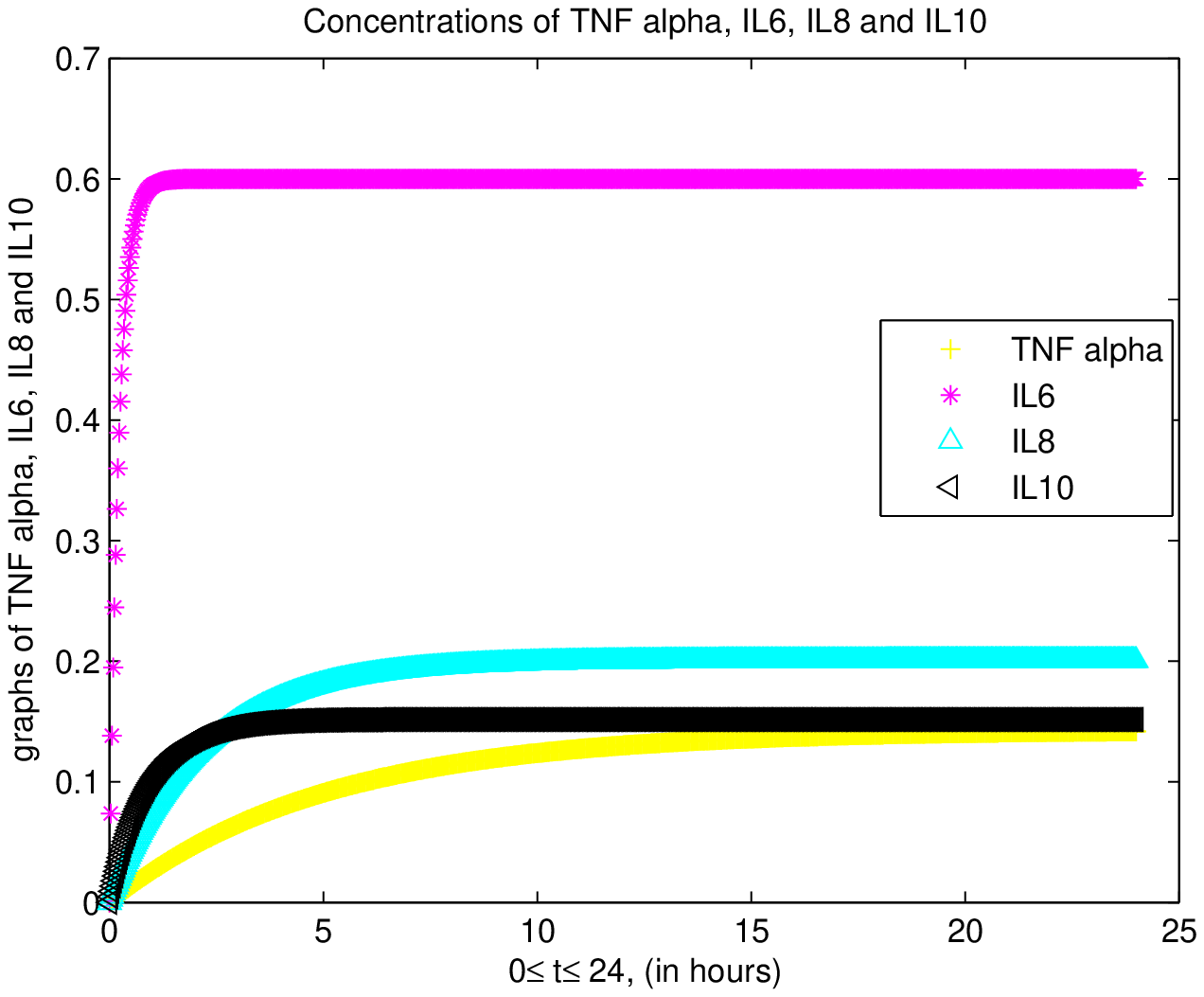,width=6cm}  & \\
         \end{tabular}
        \end{center}
        \caption{ Concentrations: S. aureus, resting macroph., activated macroph., TNF alpha, IL6, IL8 and IL10}
          \label{fig2}
          \end{figure}

          \begin{figure}
         \begin{center}
         Diffusion of S. aureus, macrophages and cytokines as provided in \cite{tgq}
          \begin{tabular}{c c}
         \psfig{file=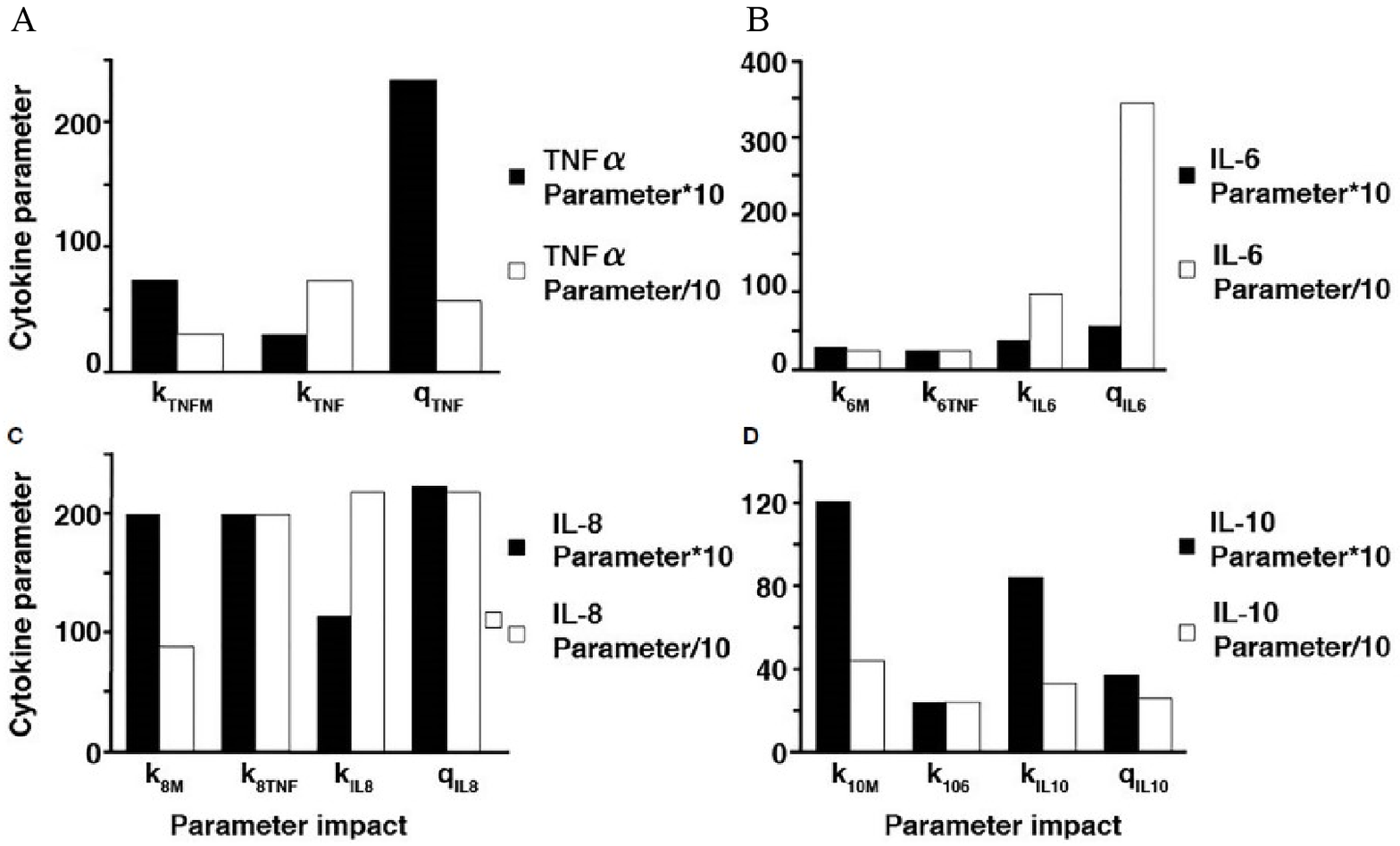,width=7cm}  & \psfig{file=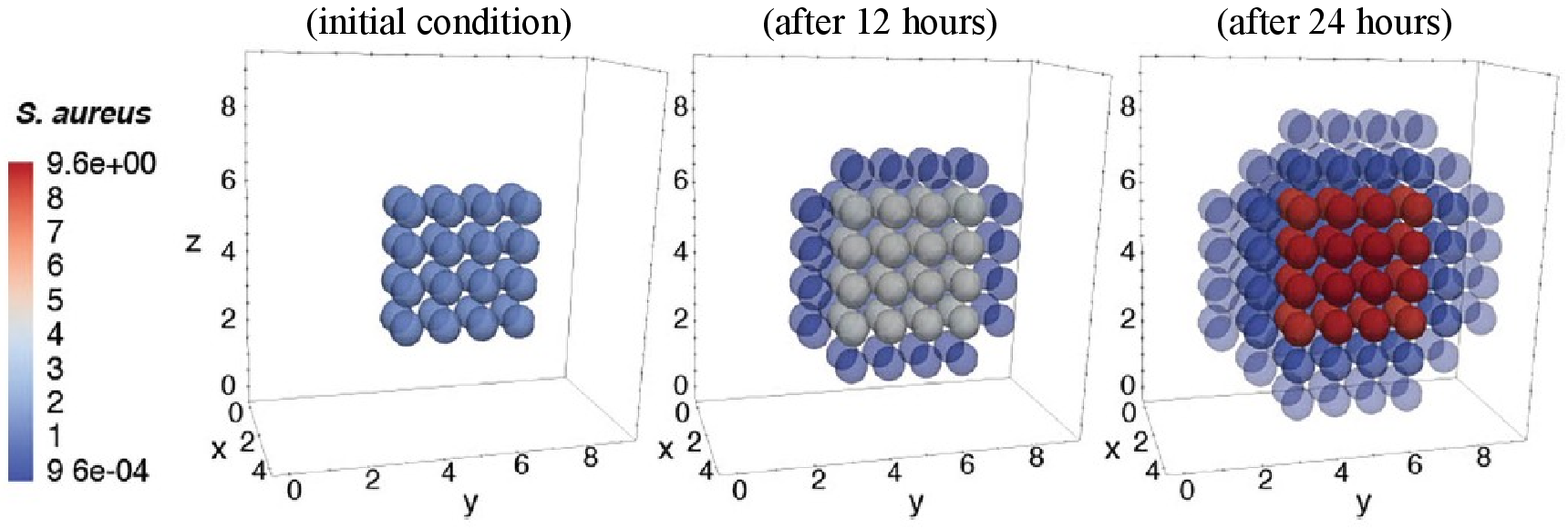,width=7cm}\\
         Parameter adjust.: TNF$\alpha$(A), IL6 (B), IL8 (C), IL10 (D)  &  S. aureus diffusion at different periods\\
         \text{\,}\\
         \psfig{file=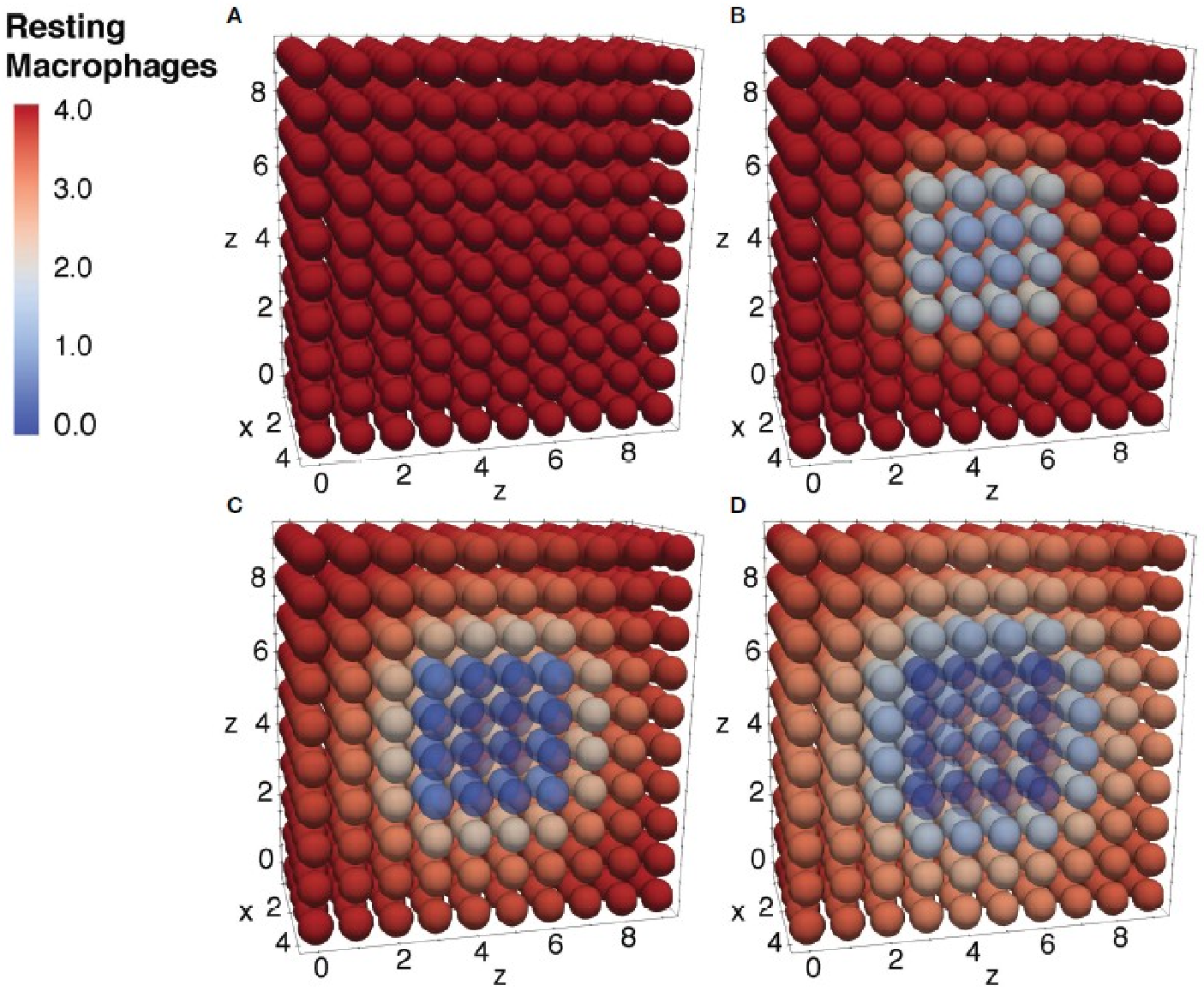,width=7cm} &  \psfig{file=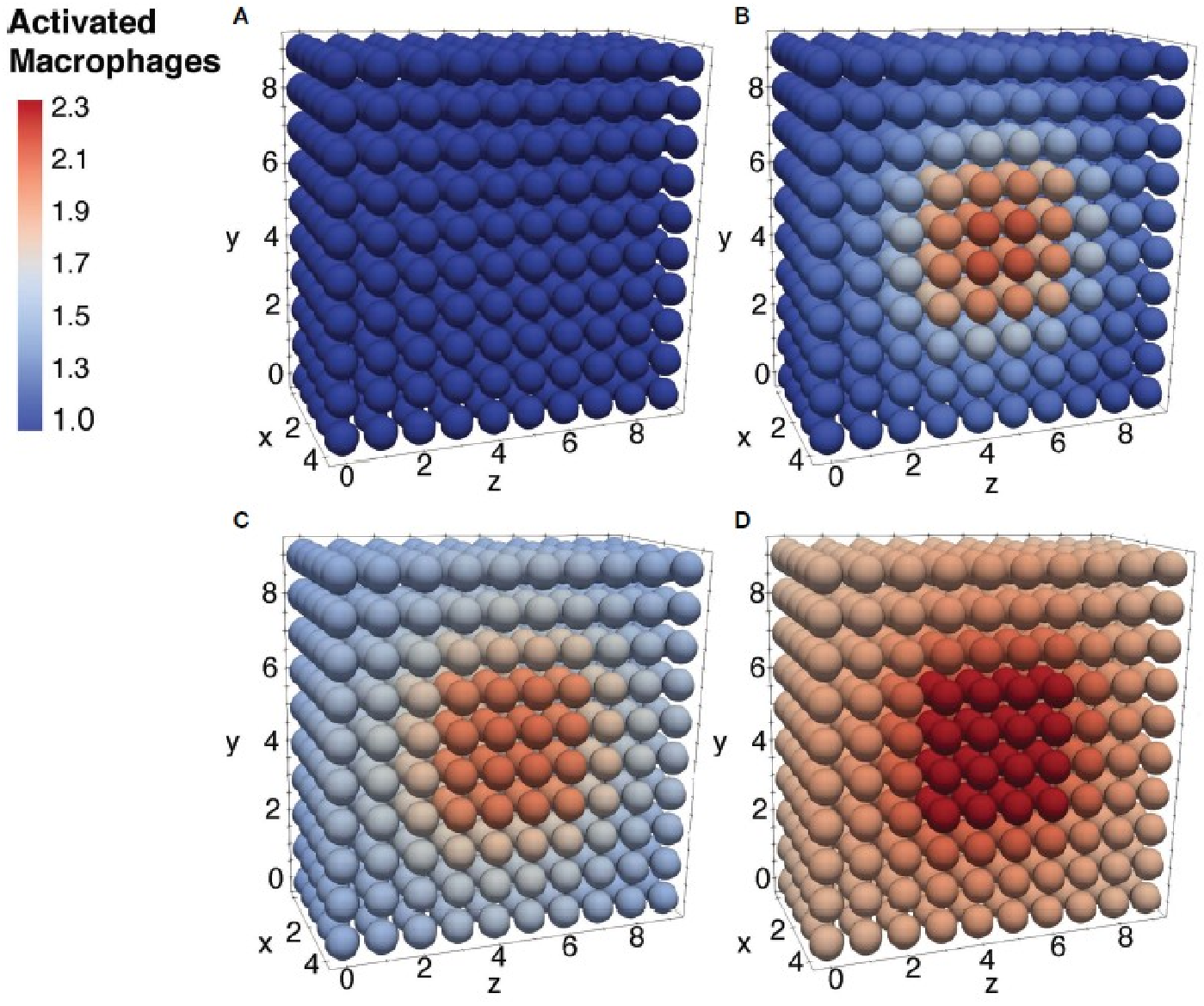,width=7cm}\\
         rest. macroph. diff.: oh(A), 3h(B), 12h(C), 24h(D)  & act. macroph. diff.: oh(A), 3h(B), 12h(C), 24h(D)\\
         \end{tabular}
        \end{center}
        \caption{ TNF alpha, IL6, IL8, IL10, S. aureus, resting macroph. and activated macroph.}
          \label{fig3}
          \end{figure}
     \end{document}